\documentclass[11pt,reqno,tbtags]{amsart}
\usepackage[]{amsmath,amssymb,amsfonts,latexsym,amsthm,enumerate,paralist}
\usepackage{amssymb,bbm,mathrsfs}
\usepackage{calc,enumerate,graphicx}
\usepackage[numeric,initials,nobysame]{amsrefs}
\usepackage[marginratio=1:1,height=584pt,width=468pt,tmargin=117pt]{geometry}

\numberwithin{equation}{section}
\newtheorem{maintheorem}{Theorem}
\newtheorem{maincoro}[maintheorem]{Corollary}

\newtheorem{theorem}{Theorem}[section]
\newtheorem*{theorem*}{Theorem}
\newtheorem{lemma}[theorem]{Lemma}
\newtheorem{claim}[theorem]{Claim}

\newtheorem{corollary}[theorem]{Corollary}

\theoremstyle{definition}{

\newtheorem*{example*}{Example}
\newtheorem{definition}[theorem]{Definition}
\newtheorem*{definition*}{Definition}

\newtheorem{remark}[theorem]{Remark}
\newtheorem*{remark*}{Remark}
}

\newcommand{\C}{\mathbb C}
\newcommand{\Q}{\mathbb Q}

\newcommand{\Z}{\mathbb Z}

\newcommand{\deq}{:=}
\newcommand{\E}{\mathbb{E}}
\renewcommand{\P}{\mathbb{P}}

\newcommand{\Bin}{\operatorname{Bin}}

\newcommand{\one}{\mathbbm{1}}

\renewcommand{\epsilon}{\varepsilon}
\renewcommand{\phi}{\varphi}

\newcommand{\cC}{\mathcal{C}}
\newcommand{\cG}{\mathcal{G}}
\newcommand{\cI}{\mathcal{I}}
\newcommand{\cJ}{\mathcal{J}}
\newcommand{\cS}{\mathcal{S}}

\newcommand{\cE}{\mathcal{E}}
\newcommand{\cF}{\mathcal{F}}
\newcommand{\cP}{\mathcal{P}}
\newcommand{\cR}{\mathcal{R}}

\newcommand{\sH}{\mathscr{H}}
\newcommand{\sM}{\mathscr{M}}
\newcommand{\sN}{\mathscr{N}}
\newcommand{\tI}{\tilde{I}}

\newcommand{\tW}{\tilde{W}}
\newcommand{\tZ}{\tilde{Z}}

\newcommand{\hit}{\sH}
\newcommand{\rod}{\textsc{rod}}
\newcommand{\comb}{\text{\sffamily\slshape Comb}_{n}}
\newcommand{\combg}{\text{\sffamily\slshape Comb}}
\newcommand{\cfact}[1][k]{\text{\sffamily \slshape CF}_{#1}}
\newcommand{\match}{\mathfrak{M}}
\newcommand{\inter}{\cI}

\date{}

\begin{document}

\title{Cycle factors and renewal theory}
%{Cycle factors, Renewals and \\ the Comb Conjecture}

\author[J. Kahn]{Jeff Kahn}
\address{Jeff Kahn\hfill\break
Department of Mathematics\\
Rutgers\\
Piscataway, NJ 08854, USA.}
\email{jkahn@math.rutgers.edu}
\urladdr{}
\thanks{J.\ Kahn is supported by NSF grant DMS0701175.}
\urladdr{}

\author[E. Lubetzky]{Eyal Lubetzky}
\address{Eyal Lubetzky\hfill\break
Microsoft Research\\
One Microsoft Way\\
Redmond, WA 98052, USA.}
\email{eyal@microsoft.com}
\urladdr{}

\author[N. Wormald]{Nicholas Wormald}
\address{Nicholas Wormald\hfill\break
%Department of Combinatorics and Optimization\\
%University of Waterloo\\
%Waterloo, ON N2L-3G1, Canada.}
%\email{nwormald@uwaterloo.ca}
School of Mathematical Sciences\\
Monash University\\
Clayton, Victoria 3800, Australia.}
\email{nick.wormald@monash.edu}
\thanks{N.\ Wormald was supported in part by the Canada Research Chairs Program and NSERC and partly by an ARC Australian Laureate Fellowship.}
\urladdr{}
\begin{changemargin}{-0.28cm}{-0.28cm}
\begin{abstract}
For which values of $k$ does a uniformly chosen $3$-regular graph $G$ on $n$ vertices typically contain $ n/k$ vertex-disjoint $k$-cycles (a $k$-cycle factor)?
To date, this has been answered for $k=n$ and for $k \ll \log n$;
the former, the Hamiltonicity problem, was finally answered in the affirmative by Robinson and Wormald in 1992,
while the answer in the latter case is negative since with high probability most vertices do not lie on $k$-cycles.

Here we settle the problem completely:  the threshold for a $k$-cycle factor in $G$ as above is
$\kappa_0 \log_2 n$ with $\kappa_0=[1-\frac12\log_2 3]^{-1}\approx 4.82$. Precisely, we prove a 2-point concentration result: if $k \geq \kappa_0 \log_2(2n/e)$ divides $n$ then $G$ contains a $k$-cycle factor w.h.p., whereas if $k<\kappa_0\log_2(2n/e)-\frac{\log^2 n}n$ then w.h.p.\ it does not.
As a byproduct, we confirm the ``Comb Conjecture,''
 an old problem concerning the embedding of certain spanning trees in
 the random graph $\cG(n,p)$.

The proof follows the small subgraph conditioning framework, but the associated second moment analysis here
is far more delicate than in any earlier use of this method and involves several novel features, among them a sharp estimate for tail probabilities in renewal processes without replacement
which may be of independent interest.
\end{abstract}

\end{changemargin}
{
\baselineskip18pt\
\maketitle
}
\vspace{-0.55cm}

\section{Introduction}
An $H$-factor of a graph $G$ is a collection of vertex-disjoint copies of the graph $H$ covering all vertices of $G$.
Thresholds for the existence of $H$-factors in random graphs have been extensively studied ---
from classical works in the 1960's (for instance, perfect matchings~\cite{ER66}) to recent ones
(such as triangle-factors and the related ``Shamir's problem'' of matchings in hypergraphs~\cite{JKV}).
Here we consider the following question on $k$-cycle factors in  random regular graphs.
\begin{quote}
 \emph{(Cycle factors.)} For which values of $k=k(n)$ does a uniformly chosen 3-regular graph on $n$ vertices contain $n/k$ vertex-disjoint $k$-cycles with high probability\footnote{A sequence of events $(A_n)$ is said to hold \emph{with high probability} (w.h.p.) if $\P(A_n)\to 1$ as $n\to\infty$.}?
\end{quote}
When $k=n$ this is the Hamiltonicity problem, which was finally answered in the affirmative in 1992 by Robinson and Wormald~\cite{RW3}. At the other extreme, for $k= O(1)$ it is known~(\cites{Bollobas1,Wormald2}) that the total number of $k$-cycles in $G$ is asymptotically Poisson with bounded mean, and in particular there is no $k$-cycle factor w.h.p.\ (the total number of vertices on such cycles is uniformly bounded); moreover, the typical absence of a $k$-cycle factor extends to the range $k\ll \log n$, throughout which most vertices do not lie on $k$-cycles w.h.p. No results were known on intermediate values of $k$.

Somewhat surprisingly, a major role in our study of the above problem will be played by a question on tail probabilities of \emph{renewal processes without replacement} --- where the recurrence times from the classical setting of renewal processes, rather than being i.i.d. random variables, are drawn uniformly yet \emph{without} replacement from a finite set. We now define this question formally.

Let $X=\{x_1,\ldots,x_m\}$ be a multiset of positive integers summing to $n$. Let $(Y_i)$ be a sequence of i.i.d.\ uniform samples of the $x_i$'s (recurrence times), and let $S_t=\sum_{i=1}^t Y_i$ be its partial sum sequence (the renewal process). By the classical Renewal Theorem (due to Erd\H{o}s, Feller and Polard~\cite{EFP} in the discrete setting), the probability that the partial sums ``hit'' some integer $k$, denoted $R_k = \P( k\in\{S_1,S_2,\ldots\})$, tends to $m/n$ as $k\to\infty$ provided that ${\rm gcd}(x_1,\ldots,x_m)=1$ (see, e.g.,~\cite{Feller}, as well as the background in \S\ref{sec-complex}). We consider the following variant of $R_k$:
\begin{quote}
\emph{(Renewals without replacement.)} Let $X=\{x_1,\ldots,x_m\}$ be a multiset of positive integers summing to $n$, take a uniform permutation $\sigma \in \cS_m$ and let $S_t = \sum_{i=1}^t x_{\sigma(i)}$. What is the probability $P_k=\P(k\in\{S_1,S_2,\ldots\})$ that the partial sums hit $k\in\Z$?
\end{quote}
Our main result will hinge on sharp quantitative bounds for $P_k$ (and a variant of it called $Q_k$), including asymptotic second order terms and correct exponential tails (see Theorem~\ref{thm-renewals} below).

Going back to the main problem on cycle factors in random regular graphs, it is interesting
to compare the situation for
the Erd\H{o}s-R\'enyi random graph $\cG(n,p)$ (in which
each edge appears independently with probability $p$), where for a given $k$ one is interested
in the threshold $p_c$ at which the probability of a $k$-cycle factor is $\frac12$ (say). Here it is
often natural to expect
that  $p_c$ coincides (up to a factor (1+o(1))) with the threshold for the property that every vertex lies on a $k$-cycle.
For instance, for $k$ fixed, the latter threshold has order $n^{-\frac{k-1}k}(\log n)^{\frac1k}$;
indeed, it was shown in~\cite{JKV} that the threshold for $k$-cycle factors has the same order,
though %the closer agreement just mentioned remains conjectural.
its asymptotics remain unknown.

For the property that
every vertex lies on a $k$-cycle in the random 3-regular graph, the threshold (`` threshold" now referring  to $k$)
is at $k = (1+o(1))\log_2 n$; this follows from the fact that a given vertex has at most (and typically also roughly) $3\cdot 2^{k/2}$ vertices at distance $k/2$ from it, and edges between these vertices have probability of order $1/n$.
(With slightly more care, the same argument shows that the threshold is $\log_2 n + \log_2 \log n +O(1) $.)

We now state our main result, which settles the problem completely and shows that here
the preceding intuition is not quite correct:
the phase transition from no $k$-cycle factor to the existence of one occurs around $[1-\frac12\log_2 3]^{-1}\log_2 n \approx 4.82 \,\log_2 n$.
Furthermore, we establish a 2-point concentration result (a single point for most values of $n$).

\begin{maintheorem}\label{mainthm-cycle-factor}
Let $G$ be a random $3$-regular graph on $n$ vertices and let
\begin{equation}\label{eq-K0}
%  \tag{\ref{eq-K0}'}
  K_0(n) = \frac{1}{1-\frac12\log_2 3} \log_2(2n/e)\,.
\end{equation}
If $k \geq K_0(n)$ is a divisor of $n$ then $G$ contains a $k$-cycle factor w.h.p., and on the other hand if $k \leq K_0(n)-\frac{\log^2 n}n$ then w.h.p.\ there is no $k$-cycle factor in $G$.

Moreover, the number of $k$-cycle
factors in $G$, denoted by $\cfact$, satisfies
\begin{equation}
   \label{eq-cfk-law}
   \frac{\cfact}{\E[\cfact]} \stackrel{\mathrm{d}}{\longrightarrow} W = \prod_{j=3}^{\infty}(1+\delta_j)^{Z_j} e^{-\delta_j\E Z_j}
\quad\mbox{ as $n\to\infty$}
 \end{equation}
for any $k \geq K_0(n)$ that divides $n$, where $\delta_j = \frac{(-1)^j-1}{2^{j}}$
and the $Z_j$'s are i.i.d.\ $\mathrm{Poisson}(\frac{2^{j-1}}j)$ variables.
\end{maintheorem}

%\begin{remark*}
%Using the same methods we can extend the above theorem to yield the threshold (at $K_0 \log_2 n$ for the same $K_0$)
%for containing $\lfloor n/k\rfloor$ vertex-disjoint $k$-cycles in a random $3$-regular graph on $n$ vertices.
%\end{remark*}
%\begin{remark*}
%We in fact prove a slightly more general result, the analogue of Theorem~\ref{mainthm-cycle-factor} when the graph $G$ is a uniformly chosen 3-regular \emph{multigraph} on $n$ vertices conditioned to have no self-loops (from which the case where $G$ is a uniform 3-regular graph can be readily derived).
%\end{remark*}
\begin{remark*}
The proof technique extends, with very few modifications, to yield the threshold (as well as a 2-point concentration) for $k$-cycles factors
in a random $d$-regular graph for any fixed $d\geq 3$.
\end{remark*}

The proof of~\cite{RW3} that a random cubic graph is Hamiltonian introduced the \emph{small subgraph conditioning method}, an interesting twist on the second moment method: upon calculating the second moment of $H_n$, the number of Hamilton cycles in that random graph, one finds that it is unfortunately (just barely) too large, namely
$\E H_n^2 / (\E H_n)^2 \to c$ for fixed $c>0$. The culprit turns out to be the set of small cycles (those with bounded length) in the graph, which in a sense blow up the variance by allowing local detours along a Hamilton cycle. Luckily --- and quite mysteriously (in various situations this fails, e.g., when half the degrees are 3 and half are 4) --- the second moment drops to $\epsilon (\E H_n)^2$ once we \emph{condition on the joint cycle distribution} up to length $M(\epsilon)$, implying Hamiltonicity with high probability (see~\cite{Wormald} and~\cite{JLR}*{\S9.3} for more information).

As our result gives Hamiltonicity for the special case $k=n$, naturally we follow the framework of small subgraph conditioning, which was highly nontrivial already for a single cycle. Though far more delicate to carry out in our setting (as explained below), this method enjoys two byproducts (implicit in~\cite{RW3} and formalized in~\cites{Janson} (and for (b) below also in~\cite{MRRW}); cf.~\cite{JLR}*{\S9.5},\cite{Wormald}*{\S4}): (a) it gives the limiting law of the variable (as in Eq.~\eqref{eq-cfk-law} above), and (b) it
further implies \emph{contiguity}.
\begin{definition*}[Contiguity of distributions]
  Let $\P_n$ and $\Q_n$ be two sequences of probability measures defined on the same measurable spaces $(\Omega_n,\cF_n)$. We say that $\P_n$ and $\Q_n$ are contiguous, denoted $\P_n\approx\Q_n$, if for any sequence of events $(A_n)$ one has $\lim_{n\to\infty}\P_n(A_n)=1 \;\Leftrightarrow\;\lim_{n\to\infty}\Q_n(A_n)=1$.
\end{definition*}
In other words, events hold w.h.p.\ under $\P_n$ iff they hold w.h.p.\ under $\Q_n$. The first example of contiguity in our context (\cites{Janson,MRRW}) was that $\cG(n,3)$, the uniform distribution on 3-regular graphs on $n$ vertices, is contiguous to the union of a uniform Hamilton cycle and a uniform perfect matching (conditioned on no multiple edges in the union).
Our work extends this result: roughly put, the next corollary says that one can distinguish with probability $1-o(1)$ between $\cG(n,3)$ and the union of a $k$-cycle factor and a perfect matching if and only if $k \geq K_0(n)$ from Eq.~\eqref{eq-K0}.
\begin{maincoro}\label{maincoro-contig}
Let $\cG(n,3)$ be the uniform $3$-regular graph on $n$ vertices and, for $k\mid n$, let $\cG(n,k,3)$ be
the union of a uniform $k$-cycle factor and a uniform prefect matching, conditioned on no multiple edges.
If $k \geq K_0(n)$ then $\cG(n,3) \approx \cG(n,k,3)$ whereas if $k\leq K_0(n)-\frac{\log^2 n}n$ then $\cG(n,3)\not\approx\cG(n,k,3)$.
\end{maincoro}

The most challenging hurdles in the proofs of these results arise in the second moment calculation, already before the small subgraph conditioning enters the picture.
Indeed, calculating the second moment of the number of Hamilton cycles amounts to understanding the typical intersection of two cycles (a collection of paths): obtaining all Hamilton cycles that contain these paths amounts simply to orienting each path and ordering the sequence of paths, i.e., stitching them into an $n$-cycle (one can verify that these paths are a partition of all the vertices since the graph is 3-regular).

However, for $k$-cycle factors, the common intersection of two such factors must be stitched into $n/k$ cycles, and now one seeks only those permutations of the $m$ parts that form $k$ cycles: letting $X=\{x_1,\ldots,x_m\}$ be the set of path lengths, we see the connection to the above question on renewals without replacement, as we wish to hit all multiples of $k$ with the partial sums of the permutation.
It is further seen that very sharp error estimates are needed, up to the correct exponential error-term; e.g., when $k\asymp \log n$, we must repeatedly hit $k$ for about $\exp(c k)$ times (to build $n/k$ cycles), along which these errors accumulate. Our next result establishes such estimates.

Recall that for a set $X=\{x_1,\ldots,x_m\}$ of positive integers summing to $n$, we let $P_k$ be the probability that $k$ belongs to the set of partial sums $S_t=\sum_{i=1}^t x_{\sigma(i)}$, where $\sigma\in\cS_m$ is uniform.
For a reason to be later explained, knowing $P_k$ would not suffice for deriving the asymptotic threshold $K_0(n)$, due to a second-order term of order $1/m$ in this probability which destroys our control over the second moment at some $k=O(\log n) $ still beyond above the desired threshold. Fortunately, our sampling procedure is a variant of the above, in which this second-order term vanishes:
\begin{compactitem}
  \item let $\sigma(1)$ be a \emph{size-biased} sample of the elements, i.e., $\P(\sigma(1)=j) = x_j/n$;
  \item let $(\sigma(2),\ldots,\sigma(m))$ be a uniform permutation in $\cS_{m-1}$ over the remaining elements.
\end{compactitem}
Let $Q_n$ be the probability that $k\in\{S_1,S_2,\ldots\}$ for this process, which we refer to as the size-biased renewal process (for brevity, while stressing that only the first step is size-biased).

\begin{maintheorem} \label{thm-renewals}
Assume that   $k\to\infty$ and $k=o(\sqrt m)$ as $n\to\infty$.  Let $R>1$  and let $g(z)$ be a power series absolutely convergent for $|z|\le R$  with $g(z) \ne 1$ for $|z|\le R$ whenever $z\ne 1$, $z\in \C$. Also let $w(n)=o(1)$ as $n\to\infty$.  Then for any $\epsilon'>0$ there exist functions
\[
q_1(n) =  o(m^{-1}) +O(R^{-k}+k^4/m^2)\quad\mbox{and}\quad
q_2(n)  =  o(m^{-1}) + O\big((R-\epsilon')^{-k}+ k^4 / m^2\big)
\]
such that the following holds.
 Let $x_1,\ldots,x_m$, $P_k$, $Q_k$ be as above, let $f(z)= \sum_{\ell} p_\ell z^\ell$ be the probability generating function of the (relative) frequencies $p_\ell = \frac1m\#\{j : x_j=\ell\}$, and assume that $|f(z)- g(z)|+ |f'(z)- g'(z)|<w(n)$ for all $|z|\le R$.
\begin{enumerate}[(a)]
\item (Renewal without replacement.) Provided that    $|f''(1)- g''(1)|<w(n)$,
$$
 \Big|P_k - \frac{m}n +    \frac{g'(1)- g'(1)^2+ g''(1)}{g'(1)^3m} \Big| \le q_1 (n)\,;
$$
\item (Renewal without replacement, size-biased.)
\[
\Big|Q_k - \frac{m}n \Big|   \le q_2(n)\,.
\]
\end{enumerate}
\end{maintheorem}

\begin{example*}[geometric distribution]%\label{ex:renewals-geom}
If $x_1=0$ and $ x_\ell/m \sim 2^{1-\ell} $ for $\ell\geq 2$, then the function $f(z)=\sum p_\ell z^\ell$ is approximately
$g(z) = z^2/ (2-z)$, which satisfies
$
[g'(1)- g'(1)^2+ g''(1)]/g'(1)^3 = \frac{2}{27}
$; our results then imply, for instance, that for any $1 \ll k \ll m^{1/4}$ and any fixed $\epsilon>0$ we have
\begin{align*}
 R_k &= \tfrac13 + O\big((2-\epsilon)^{-k}\big) &\mbox{\em (with replacement)}\,,\\
  P_k &= \tfrac{1}3 - \tfrac{2/27-o(1)}m + O\big((2-\epsilon)^{-k}\big) &\mbox{\em (without replacement)}\,,\\
Q_k &= \tfrac{1}3 - o(1/m) + O\big((2-\epsilon)^{-k}\big) &\mbox{\em (size-biased without replacement)}\,.
\end{align*}
This  example will be fundamental for the proof of Theorem~\ref{mainthm-cycle-factor}. Recall that the $x_i$'s correspond to the lengths (in vertices) of the paths that comprise the common intersection of two $k$-cycle factors.
First, every such path has at least two vertices, whence $p_1 = 0$. Second, heuristically, suppose we are given
 a $k$-cycle factor $F_1$ and construct another, $F_2$, via a simple random walk (ignoring the many dependencies that exist in reality). While this random walk traverses on a common edge, there is a probability of $1/2$ that the next edge will extend the common path (i.e., follow the trace of $F_1$), and hence the geometric distribution with this parameter.
Note that for $k\geq (2+\epsilon')\log_2 n$ the error-terms are all $o(1/n)$, which will be crucial to our arguments.
\end{example*}

\subsection{Applications for the Erd\H{o}s-R\'enyi random graph}

An immediate corollary of Theorem~\ref{mainthm-cycle-factor} is that, for $k\geq K_0(n) \approx 4.82\, \log_2 n$, the threshold for a $k$-cycle factor in the random graph $\cG(n,p)$ has order $\frac{\log n}n$ (see Corollary~\ref{cor-cycle-factor-gnp}), with a factor $2-o(1)$ between the upper and lower bounds.

\smallskip

Another corollary is the ``Comb Conjecture.''
A {\em comb of order k}, for some $k\mid n$, is a tree consisting
of an $(n/k)$-vertex path $P$
together with disjoint $k$-vertex paths beginning at the vertices of $P$.
When $k=\sqrt{n}$ we will call this {\em the comb} and denote it $\comb$.
The ``Comb Conjecture" says that the threshold for the appearance of $\comb$ in $\cG(n,p)$ has
order $\frac{\log n}n$ (the lower bound is obvious due to connectivity).

Consideration of the threshold for combs was suggested by the first author about 20 years ago as a
test case for the more general conjecture (this was also proposed at that time by the first author,
but, being a natural guess, is perhaps better regarded as folklore),
that the same threshold statement holds for general ($n$-vertex) trees of bounded degree.
The case $k=\sqrt{n}$ of this suggestion has come to be known as the ``Comb Conjecture."
The rationale for considering combs was the idea that they interpolated between instances for which the
general conjecture was known to be true, namely
Hamilton paths (for which the threshold $(1+o(1))\frac{\log n}n$ was established in~\cite{KS} and~\cite{Bollobas-1})
and trees with order $n$ leaves (which are easily handled via Hall's Theorem).

For related work on this topic, see, e.g.,~\cite{AlKrSu}, which shows that, already at $p=O(\frac1n)$, w.h.p.\ the random graph contains \emph{every} bounded-degree tree on $(1-\epsilon)n$ vertices, with an implicit constant depending on $\epsilon$ and $\Delta$ (see also the refined bounds in~\cite{BCPS}); \cite{BW} which shows that $\cG(n,\frac{c\log n}n)$ w.h.p.\ contains almost every tree on $n$ vertices; and \cite{Krivelevich}, which shows that the threshold for any bounded-degree tree $T$ is at most $n^{-1+o(1)}$. The latter was achieved by observing that such trees have either many leaves or long paths of degree-2 vertices, and then deploying a separate strategy in each case. The comb is indeed an extremal example as it precisely balances between these two elements.

\smallskip
We confirm the Comb Conjecture, and obtain the threshold of $\comb$ up to a factor of $2+o(1)$.
\begin{maintheorem}\label{mainthm-comb}
%Let $\comb$ denote the $n$-vertex comb, i.e.\ the tree consisting of a $\sqrt{n}$-vertex path with disjoint $\sqrt{n}$-vertex paths rooted at each of its vertices.
For any $\epsilon > 0$ the Erd\H{o}s-R\'enyi random graph $\mathcal{G}(n,p)$ with $p=(2+\epsilon)\frac{\log n}n$ contains a copy of $\comb$
as a spanning subgraph w.h.p. In particular, the threshold for the appearance of $\comb$ in $\mathcal{G}(n,p)$
is at $p \asymp \frac{\log n}n$.
\end{maintheorem}
More generally, we get to within a factor of $2+o(1)$ of the threshold
for containing the comb of order $k$ for every $k\geq K_0(n)$ (see Remark~\ref{rem-gen-comb}).
In the companion paper~\cite{KLW} we treat the complementary range of $k$ and conclude that for any $k=k(n)$ the threshold is $O(\frac{\log n}n)$.

%\smallskip
%As a final corollary, one can apply the $k$-cycle result of Theorem~\ref{mainthm-cycle-factor} in the context of the 2-core of the giant component of $\cG(n,p)$ as soon as it emerges from the critical window; namely, $p = \frac{1+\epsilon}n$ for $n^{-1/3} \ll \epsilon \ll n^{-1/4}$. In that range, for instance, it is still unknown what the asymptotic length of the longest cycle is (its order is known~\cite{Luczak91} to be $\epsilon^2 n$). Nevertheless, by the contiguity result in~\cite{DKLP}, it

%Remark on $k$-cycle factor in the 2-core of $G(n,\frac{1+\epsilon}n)$ for $n^{-1/3} \ll \epsilon \ll n^{-1/4}$?

\subsection{Notation and organization}
On occasion we will write $f_n \lesssim g_n$ instead of $f_n=O(g_n)$ for brevity (similarly for $f_n \gtrsim g_n$); $f_n \sim g_n$ denotes $f_n = (1+o(1))g_n$, and $f\asymp g$ denotes $f_n \lesssim g_n \lesssim f_n$.

The rest of this paper is organized as follows. In \S\ref{sec-renewals} we establish Theorem~\ref{thm-renewals}, among other results on renewal processes with and without replacement. Sections~\ref{sec-second-moment-framework}--\ref{sec:h=0-reduction} are devoted to the analysis of the second moment of the number of $k$-cycle factors. The application of the small subgraph conditioning method (and its consequences for contiguity) appears in \S\ref{sec-small-subgraph}, and concludes the proof of Theorem~\ref{mainthm-cycle-factor}, as well as Corollary~\ref{maincoro-contig}. Section~\ref{sec-gnp-comb} contains the proof of the Comb Conjecture (Theorem~\ref{mainthm-comb}).

\section{Renewal processes with and without replacement}\label{sec-renewals}
In this section we obtain various results on the probability of the event $\hit_k$ that a renewal process $(S_t)$ hits a given value $k$ (i.e., $S_t=k$ for some $t$), which in particular establish Theorem~\ref{thm-renewals} that will be used in our later arguments; see \S\ref{sec-intersection}.
%
%Suppose that we generate the variable  $Y_j$ by sampling a random number from a multiset $X=\{x_1, \ldots, x_m\}$ u.a.r.  In the standard renewals setting, which we refer to as \emph{with replacement}, the selected number $x_\ell$  is replaced before the experiment is repeated; equivalently, $Y_j=x_{\sigma(j)}$ where the $\sigma(j)$ are independently selected u.a.r.\ from $[n]$. Here we are mainly concerned with   \emph{renewals without replacement}, where the $Y_j$ are not i.i.d., but instead the element $x_\ell$ selected for $Y_j$ is discarded from $X$   before the later $Y_r$, $r>j$, are generated. Thus,  in this case $\sigma$ is  a random \emph{permutation}  of $[m]$; as before, the $S_i$ is the $i$-th partial sum of the sequence $x_{\sigma(1)},\ldots, x_{\sigma(m)}$. Shifting to renewals without replacement naturally causes a slight change in the value of $\P(\hit_k)$, even when $k$ is small relative to $m$. We will actually be interested mainly in the value of $\P(\hit_k)$ in a slightly related model, in which the probability of the first-chosen element  of $X$ is altered by a factor proportional to its value. That is, $\P(Y_1=x_\ell)=x_\ell/\sum x_j$. The remaining choices constitute a uniformly random permutation of $X\setminus \{x_\ell\}$, where $Y_1=x_\ell$.  We call this the \emph{size-biased} case of renewals without replacement.
We first state the main results of this section.  Note that we will require estimates of $\P(\hit_k)$ in which the error terms are uniform over a range of sets $X$, which seems to require  a new result even in the standard setting.

We approach the question by coupling the two models, with and without replacements,  in Part~(b) of the following theorem. Here, $[z^k]$ denotes extraction of the coefficient of $z^k$. Theorem~\ref{thm-renewals-estimates} will provide estimates for the coefficients. We assume that all variables are functions of $n$.
\begin{theorem}
  \label{thm-renewals-connection}
  Let   $x_1,\ldots,x_m$ be positive integers with $\sum_{j=1}^m x_j = n$.
Write $p_\ell = \frac1m |\{j \,:\, x_j=\ell\}|$ and define  $f(z)= \sum_{\ell\ge 1} p_\ell z^\ell$. For $\sigma:[m]\to [m]$ define
\[
Y_t =Y_t(\sigma) = \sum_{j \leq t} x_{\sigma(j)} \quad\mbox{for $t=1,\ldots,m$}\,,
\]
and let  $R_k$ equal $\P\big(k \in \{Y_1,\ldots,Y_m\}\big)$ when   $\sigma$ is   selected u.a.r., and let   $P_k$ equal $\P\big(k \in \{Y_1,\ldots,Y_m\}\big)$ conditional upon $\sigma$ being a    permutation of $[m]$. Then
\begin{enumerate}[(a)]
\item $R_k= [z^k] \big(1-f(z)\big)^{-1}$;
\item as $n\to\infty$, provided that $k=o(\sqrt m)$ we have
\[
 P_k = R_k -\big(m^{-1} +O(k^2/m^2)\big)[z^k]\frac{f(z^2)-f(z)^2}{\big(1-f(z)\big)^3}  +O(k^4/m^2)\,.
\]
 Here, the constants implicit in the $O(\cdot)$ are absolute, in particular not depending on the $x_j$'s.
\end{enumerate}
\end{theorem}

The  list of constants and conditions in the following theorems is made longer than might be expected, because of our need  to apply them uniformly to a class of sets of numbers $\{x_j\}$.
\begin{theorem}
  \label{thm-renewals-estimates}
Define $m$, $n$ and $f(z)$ as in Theorem~\ref{thm-renewals-connection}.
Suppose that for some real constants $r>1$, $c>0$ and $c_0>0$, we have
\begin{enumerate}[(i)]
\item  $\sum_{\ell\ge 1} p_\ell r^\ell \le c $,
\item $f(z) \ne 1$ for $|z|\le r$ whenever $z\ne 1$, $z\in \C$, and
\item $\min_{|z|=r} |f(z)-1|\ge c_0^{-1} $   ($z\in \C$).
\end{enumerate}
Then for all $k\ge 1$
\[
|R_k-m/n| \le  c_0 r^{-k}
\]
and
\[
\left|[z^k]\frac{f(z^2)-f(z)^2}{(1-f(z))^3}- \frac{f'(1)- f'(1)^2+ f''(1)}{f'(1)^3}\right|  \le c_1 r^{-k/2}\,,
\]
where
\[
c_1=\max_{|z|=\sqrt r} \left|\frac{f(z^2)-f(z)^2}{ (1-f(z))^3}\right|\,.
\]
\end{theorem}
Note that $f'(1)=n/m$.

After many twists and turns,  it is an implication of the proof of our main result   that the   formulae as above cannot possibly be extended to  the size-biased case without modification, because  otherwise a certain random variable involving $k$-cycle factors would have a negative variance. The negative term of order  $m^{-1}$ in $P_k$ is the problem. Fortunately, this term is cancelled out in the size-biased case, as follows.
\begin{theorem}
  \label{thm-renewals-size-biased}
Define   $x_1,\ldots,x_m$, $p_\ell$, $m$, $n$ and $f(z)$  as in Theorem~\ref{thm-renewals-connection}.
Let $J$ be the random variable given by $\P(J = j) = x_j / n$ for $j=1,\ldots,m$, let $\sigma$ be a random permutation of the indices $\{1,\ldots,m\}\setminus\{J\}$, and define
\[
\hat{Y}_t =\hat{Y}_t(\sigma) = x_J + \sum_{j \leq t-1} x_{\sigma(j)} \quad\mbox{for $t=1,\ldots,m$} \,.
\]
Set $Q_k=\P\big(k \in \{\hat{Y}_1,\ldots,\hat{Y}_m\}\big)$. Assume, for some positive constants   $r$,  $c$, $c_0$  and $\delta$   with $1+\delta<\sqrt{r }$,
\begin{enumerate}[(i)]
\item  $\sum_{\ell\ge 1} p_\ell r^\ell \le c $,
\item $|f(z)-1| \ge c_0^{-1}$ if $|z|\le r$ and $|z-1|\ge \delta$, $z\in \C$,
\item $|f'(z) -f'(1)|<f'(1)/2 $ if  $|z-1|< \delta$, $z\in \C$.
\end{enumerate}
Then as $n\to\infty$, provided that   $k\to\infty$ and   $k=o(\sqrt m)$, we have for any $\epsilon>0$
\[
\Big|Q_k - \frac{m}n \Big| \le     q(n)\,,
\]
where $q(n)$ is a function of $r$, $c$, $c_0$, $k$ and $\epsilon$ satisfying $q(n) \leq  2c_0 (r-\epsilon)^{-k} + O\left(k^4 / m^2\right) +o(1/m)$. In particular, $q(n)$ does not otherwise depend on the $x_j$'s.
\end{theorem}

%For ease of estimating the probability of hitting $k$ for renewals without replacement,  both size-biased and not, we give the following corollary of the previous three theorems.
%\begin{corollary} \label{cor-renewals}
%Assume that   $k\to\infty$ and $k=o(\sqrt m)$ as $n\to\infty$.  Let $R>1$  and let $g(z)$ be a power series absolutely convergent for $|z|\le R$  with $g(z) \ne 1$ for $|z|\le R$ whenever $z\ne 1$, $z\in \C$. Also let the function $w(n)$ satisfy $w(n)\to 0$ as $n\to\infty$.  Then there exists a function $q_1(n)$, with
%\[
% q_1(n) =  o(m^{-1}) +O(R^{-k}+k^4/m^2)
% \]
%as $n\to\infty$, and for all $\epsilon'>0$ a function $q_2(n)$, with
%\[
%q_2(n)  =  o(m^{-1}) + O\big((R-\epsilon')^{-k}+ k^4 / m^2\big)
%\]
%as $n\to\infty$, such that the following holds.
%
% Define $x_1,\ldots,x_m$, $p_j$, $m$, $n$, $P_k$ as in Theorem~\ref{thm-renewals-connection}, and $Q_k$ as in Theorem~\ref{thm-renewals-size-biased},
%and assume that $f(z)$  satisfies $|f(z)- g(z)|+ |f'(z)- g'(z)|<w(n)$ for all $|z|\le R$.
%\begin{enumerate}[(a)]
%\item (Renewal without replacement.) Provided that    $|f''(1)- g''(1)|<w(n)$,
%$$
% \Big|P_k - \frac{m}n +    \frac{g'(1)- g'(1)^2+ g''(1)}{g'(1)^3m} \Big| \le q_1 (n)\,;
%$$
%\item (Renewal without replacement, size-biased.)
%\[
%\Big|Q_k - \frac{m}n \Big|   \le q_2(n)\,.
%\]
%\end{enumerate}
%\end{corollary}

In the next subsection, we give some estimates for the generating function coefficients appearing in Theorem~\ref{thm-renewals-connection}. In \S\ref{sec-renewals-proofs} we complete the proof of all three theorems and the corollary.

\subsection{Singularity analysis}\label{sec-complex}

Assume that  $p_0=0$ and $p_\ell\ge 0$ for $\ell\ge 1$, with $\sum_{\ell\ge 0}p_\ell=1$.    Let $f(z)=\sum_{\ell\ge 1} p_\ell z^\ell$. Note that $f(1)=1$. We are interested in $u_n:=[z^n](1-f)^{-1}$ for renewals without replacement, and some coefficients in related generating functions for renewals with replacement.

In the following, $z\in \mathbb{C}$. Suppose that $R>1$, $f$ is holomorphic in $|z|< R$ (equivalently, $p_\ell=O(R^{-\ell})$), and ${\rm gcd}\{\ell:p_\ell>0\} = 1$.
 Kendall~\cite{Kendall} showed that under these conditions, $u_\infty:=\lim_{n\to\infty} u_n$ exists, and that
$\sum_{n\ge 0} (u_n-u_\infty)z^n$ has radius of convergence strictly greater than 1. It follows that  $|u_n-u_\infty|=O(r^{-n})$ for some $r>1$. This can easily be proved using the  method of `subtracting the singularity.' (See Wilf~\cite{Wilf}*{\S5.2} for a description of this method.)

Baxendale~\cite{Baxendale}*{Theorem 3.2} examined the analyticity of the function $\sum_{n\ge 1}  (u_n-u_{n-1})z^n $, and hence obtained explicit bounds on $r$ and on $\sum_{n\ge 0} (u_n-u_\infty)z^n$ ($|z|=r$) under certain conditions.  His conditions included $p_1>0$ so we cannot apply his general theorem to our case. The approach could however be adapted to our case; it is closely related to the method of subtracting the singularity.
Such explicit results can be used to obtain bounds that hold uniformly for a family of functions, which is what we desire here. We will use the alternative approach of contour integration. Flajolet and Sedgewick~\cite{SF}*{Lemma~IX.2 (p.~668)} show  how to obtain uniform estimates using this approach, but their result is not quite suitable for our current purpose. Additionally, we have the opportunity to use a simpler contour in this case.
\begin{lemma}\label{l:contour}
Let $ r>1$ and assume that
\begin{enumerate}[\indent(i)]
\item\label{item:contour:1} $f$ is holomorphic in $|z|< r$ and continuous on $|z|\le r$;
\item\label{item:contour:2} $f(z) \ne 1$ for $|z|\le r$ when $z\ne 1$.
\end{enumerate}
Then, with $u_\infty= \lim_{n\to\infty}u_n$,
\begin{enumerate}[\indent(a)]
\item\label{item:contour:a} $u_\infty= f'(1)^{-1}$, and  $\sum_{n\ge 0} (u_n-u_\infty)z^n$ has radius of convergence at least $r$.
\end{enumerate}
If in addition
\begin{enumerate}[\indent(i)]\setcounter{enumi}{2}
\item\label{item:contour:3} $\min_{|z|=r} |f(z)-1|=c_0>0$,
\end{enumerate}
then
\begin{enumerate}[\indent(a)]\setcounter{enumi}{1}
\item\label{item:contour:b} $\displaystyle
|u_n-u_\infty|\le \frac{1}{c_0r^n}.
$
\end{enumerate}
\end{lemma}
\begin{proof}
Define $g(z)=(1-f(z))/(1-z)$ as a formal power series, so
$g(z)= \sum_{n\ge 0} g_n z^n$ with
\[
g_n = 1-\sum_{1\le \ell\le n} p_\ell = \sum_{\ell\ge n+1} p_\ell\qquad\mbox{for $n\ge 0$}\,.
\]
By~\eqref{item:contour:1}, this series for $g(z)$ has radius of convergence at least $r$, and $g(1)\ge g_0=1\ne 0$.

Recall that $u_n=[z^n](1-f)^{-1}$. For $|z|<r$, $z\ne 1$, we have
\begin{equation}\label{equiv}
\frac{1}{1-f(z)}=\frac{1}{g(z)(1-z)}.
\end{equation}
The method of subtracting the singularity now gives the conclusion on the radius of convergence in (a), but we omit details  as we need the more precise bound in (b). For this, we use the Cauchy integral formula to extract the coefficient of $z^n$.

Let $\cC$ be a contour that passes around the circle $|z|=r $ in a counterclockwise direction beginning at $R$, then along the real line from $z=r$ to $z=1+\epsilon$ (for some $\epsilon>0$), then once around the circle $|z-1|=\epsilon$ clockwise, then back along the real line from $z=1+\epsilon$ to $z=r$. By~\eqref{item:contour:1} and~\eqref{item:contour:2}, $(1-f)^{-1}$ is holomorphic on the interior of $\cC$ and continuous on its closure, so
\[
[z^n](1-f(z))^{-1} = \frac{1}{2\pi i}\oint_\cC \frac{z^{-n-1}}{ 1-f(z)}\, dz.
\]
First, note that $(1-f(z))^{-1}$ is bounded above on the outer circle since $1-f$ is   non-zero there (and $f$ is continuous). Hence, that part of the contour integral on the outer circle is $O(r^{-n})$. In fact, under assumption~\eqref{item:contour:3}, its absolute value is at most $(2\pi r)r^{-n}/c_0$. The part on the straight lines cancels since~\eqref{equiv} gives a unique value for the function there. So
$$
\frac{1}{2\pi i}\oint_\cC \frac{z^{-n-1}}{ 1-f(z)}\, dz = O(r^{-n})- \frac{1}{2\pi i}\oint_{|z-1|=\epsilon} \frac{z^{-n-1}}{ 1-f(z)}\, dz
$$
where the direction of integration is counterclockwise.
Since $g$ is analytic near 1 and $g(1)\ne 0$,  we may expand  $1/g(z)$ as a power series about $z=1$, and we see that the integrand has a simple pole at $z=1$ with principal part  $ (g(1) (1-z))^{-1} $. So
$$
[z^n](1-f(z))^{-1} = O(r^{-n}) -\mathrm{Res}_1 \frac{1}{(1-z)g(1)} = O(r^{-n})+g(1)^{-1}.
$$
It follows that $u_\infty=g(1)^{-1}$. Noting that
$$
g(1)= \sum_{n\ge 0} \sum_{\ell\ge n+1} p_\ell =\sum_{\ell\ge 1}\ell p_\ell=f'(1),
$$
we deduce~\eqref{item:contour:a}. Recalling the above given explicit bound  on the error term  $O(r^{-n})$ yields~\eqref{item:contour:b}.
 \end{proof}

For $f$ satisfying~\eqref{item:contour:1}, Lemma~\ref{l:contour} shows that the radius of convergence of $(1-f)^{-1}$ will be at least  the smallest   value of $|z|\le r$ where $z\ne 1$ and $f(z)=1$ (if any such $z$ exist). It seems reasonable to suppose that this will be equality, though our argument does not prove this.

For the case of renewals with replacement, we will make use of the following, which again uses     essentially standard singularity analysis.
\begin{lemma}\label{l:fsquared}
Assume that $f$  satisfies conditions~\eqref{item:contour:1} and~\eqref{item:contour:2} of Lemma~\ref{l:contour}. Then
\[
[z^n]\frac{f(z^2)-f(z)^2}{(1-f(z))^3}= \frac{f'(1)- f'(1)^2+ f''(1)}{f'(1)^3} +O(r^{-n})\,.
\]
Moreover, the $O(r^{-n})$-term is in absolute value at most $c_1 r^{-n/2}$ where
the constant $c_1$ is given by
\[
c_1=\max_{|z|=\sqrt r} \left|\frac{f(z^2)-f(z)^2}{ (1-f(z))^3}\right|\,.
\]
\end{lemma}
\begin{proof}
Note that $c_1$ exists by condition~\eqref{item:contour:2} of Lemma~\ref{l:contour} and the continuity of $f$.
 Since $f(z)$ is analytic near $z=1$, and $f(1)=1$,  in any small neighborhood of $1$ it can be expanded as
\[
f(z)=1 +f'(1)(z-1)+\frac12 f''(1)(z-1)^2 +O((z-1)^3)
\]
and we have
\begin{eqnarray*}
f(z^2)&=&1 +2f'(1)(z-1)+ (f'(1)^2+ f''(1))(z-1)^2 +O((z-1)^3)\,,\\
f(z)^2& =&1 +2f'(1)(z-1)+ (f'(1) + 2f''(1))(z-1)^2 +O((z-1)^3)
\end{eqnarray*}
in the same neighborhood.
We evaluate the coefficient following the proof of Lemma~\ref{l:contour}, using almost the same contour: define $\cC'$ the same as $\cC$ but with   $|z|=r$ replaced by  $|z|=\sqrt r$. Then
\[
[z^n]\frac{f(z^2)-f(z)^2}{(1-f(z))^3} = \frac{1}{2\pi i}\oint_{\cC'} \frac{f(z^2)-f(z)^2}{ (1-f(z))^3z^{n+1}}\, dz\,.
\]
The part of the integral on  $|z|=\sqrt r$ is bounded above in absolute value by $c_1/r^{n/2}$.
Near $z=1$, after expanding $1/g(z)^3=1/g(1)^3+O(z-1)$ and using the above expansions, we see the integrand has a simple pole, with residue
$-\big(f'(1)- f'(1)^2+ f''(1)\big)/  f'(1)^3$. The result follows.
\end{proof}

\subsection{Proofs of renewal results}\label{sec-renewals-proofs}

Here we are interested in the hitting probability $\P(\hit_k)$ (where $\hit_k=\bigcup_{j\ge 0} \{S_j=k\}$) for the above defined renewal sequence $(S_i)_{i\ge 1}$ without replacement.

\begin{proof}[\textbf{\em Proof of Theorem~\ref{thm-renewals-connection}}]
Since the $x_j$ are positive integers, to transfer results on $\P(\hit_k)$ from `with replacement' to `without replacement', we can restrict ourselves to considering the first $k$ holding times, $Y_1,\ldots , Y_k$,   regarded as a random sequence. We use $\nu$ to denote the probability measure in the case of renewals with replacement, where  each $Y_j=x_{\sigma(j)}$ is independently chosen u.a.r.\  from $[m]$, and $\pi$ in the case that $\sigma$ is a random permutation of $[m]$.
Define a \emph{duplicate} in $\sigma$ to be a pair $(i,j)$, $i<j\le k$, for which $\sigma(i)=\sigma(j)$, and let   $D$ denote the number of   duplicates. Then
\begin{equation}\label{total}
\nu(\hit_k)=\nu(\hit_k\mid D=0)\nu(D=0)+  \nu(\hit_k\mid D\ge 1)\nu(D\ge 1)\,,
\end{equation}
and clearly
$$
\nu(\hit_k\mid D=0)  = \pi(\hit_k)\,.
$$
Note that $\E_{\nu}{D} =\binom{k}{2}/m$ and $\E_{\nu}\binom{D}{2} =O(k^4/m^2)$ which imply by inclusion-exclusion that (recalling $k=o(\sqrt m)$ from the theorem's hypothesis)
\begin{equation}\label{D1andD2}
\nu(D=1)=\binom{k}{2}m^{-1}+O(k^4/m^2)\,,\qquad \nu(D\ge 2)=O(k^4/m^2)\,.
\end{equation}

To make use of~\eqref{total}, it only remains to estimate $\nu(\hit_k\mid D\ge 1)$. For this, we define a modified probability space in which the probability of $\sigma$ is weighted by $D(\sigma)$. Let $\Omega=\{(\sigma, i,j): (i,j)\mbox{ is a duplicate in }\sigma\}$, and endow $\Omega$ with the uniform probability measure. By symmetry, to generate $(\sigma, i,j)$ at random from $\Omega$,  we can first select $(i,j)$ u.a.r.\ from $[k]$ and $r$ u.a.r.\ from $[m]$,  then set $\sigma(i)=
\sigma(j)=r$,  and finally generate the rest of $\sigma$ by sampling independently from $[m]$ as with $\nu$. Let $\mu$ denote this probability measure on $\Omega$.
Considering this generation procedure, and noting that it ``plants'' a duplicate, it is immediate that
\[
\mu(D\ge 2)\le \E_{\mu}\binom{D}{2} =O(k^2/m) =o(1)
\]
since the expected number of pairs of non-planted duplicates is $O(k^2/m )$  and  $k=o(\sqrt m)$. Noting the equivalence of $\mu$ and $\nu$ conditioned on the event $D=1$, this implies
\begin{align}
\mu(\hit_k)&=\mu(\hit_k\mid D=1)+O(k^2/m)=\nu(\hit_k\mid D=1)+O(k^2/m)
%\nonumber\\ &
=\nu(\hit_k\mid D\ge1)+O(k^2/m) \nonumber
\end{align}
in view of~\eqref{D1andD2}.
Substituting this into~\eqref{total}, we find
\begin{align}
 \pi(\hit_k) &= \nu(\hit_k\mid D=0)  = \frac{ \nu(\hit_k)-\nu(\hit_k\mid D\ge 1)\nu(D\ge 1) }{ \nu(D=0)}\nonumber \\
&=  \frac{ \nu(\hit_k)-\big(\mu(\hit_k )+O(k^2/m)\big)\nu(D\ge 1) }{ 1-\nu(D\ge 1)}\,.
 \label{mu-Hk}
\end{align}

 So we may concentrate on estimating $\mu(\hit_k)$.
Define $T$ on $\hit_k$ such that $\sum_{s=1}^T x_{\sigma(s)}= k$, and partition the event $\hit_k$ in $\Omega=\{(\sigma, i,j)\}$ into events $A$, $B$ and $C$ as follows.
\[
A = \{  i<j\le T\}\,,\quad B = \{ i\le T<j\}\,, \quad C = \{ T<i<j\}\,.
\]
\newcommand{\remove}[1]{}
\remove{
\begin{align*}
\mbox{in }A: \quad &\sum_{s=1}^j x_{\sigma(s)} \le k\\
\mbox{in }B:  \quad &\sum_{s=1}^i x_{\sigma(s)} \le k\mbox{ \ and \ } \sum_{s=1}^j x_{\sigma(s)} > k \\
\mbox{in }C: \quad  &\sum_{s=1}^i x_{\sigma(s)} > k\,.
\end{align*}
}
Letting $\hit_k(t)$ denote the event that $\sum_{s=1}^t  x_{\sigma(s)}=k$, we can sum over $t$ and the value $\ell$ of $x_{\sigma(i)}$ to get
$$
\binom{k}{2} \mu(A) =  \sum_{t\ge 1}\sum_{l\ge 1}\binom{t}{2}\nu\big(\hit_{k-2\ell}(t-2)\big) p_{\ell}\,.
$$
Here the prefactor arises because   $\binom{k}{2}^{-1}$ is the probability of picking the pair $(i,j)$ in the above description of generating of a random element of $\Omega$.
It follows by elementary considerations that
\[
\mu(A) = \binom{k}{2}^{-1}[z^k]\frac{f(z^2)}{\big(1-f(z)\big)^3}
\]
with $f$ as in Theorem~\ref{thm-renewals-connection}.
Similarly,
$$
\binom{k}{2}\big(\mu(B) +\mu(C)\big)=  \sum_{t\ge 1}\sum_{l\ge 1}t(k-t)\nu\big(\hit_{k- \ell}(t-1)\big) p_{\ell}+ \binom{k-t}{2}\nu\big(\hit_{k }(t )\big) p_{\ell}\,.
$$
Rewriting $\binom{k-t}{2}$ as $\binom{k}{2}-t(k-t)-\binom{ t}{2}$  and using
$$
\sum_{l\ge 1} \nu\big(\hit_{k- \ell}(t-1)\big) p_{\ell}= \nu\big(\hit_{k }(t )\big)=\sum_{l\ge 1} \nu\big(\hit_{k }(t )\big) p_{\ell}\,,
$$
we find that the terms containing $t(k-t)$ cancel. The remaining ones are
$$
 \sum_{t\ge 1}\sum_{l\ge 1}\binom{k }{2} \nu\big(\hit_{k }(t )\big) p_{\ell}=
  \sum_{t\ge 1} \binom{k}{2} \nu\big(\hit_{k }(t )\big) = \binom{k}{ 2}\nu (\hit_{k } )
$$
and
\begin{align*}
-\sum_{t\ge 1}\sum_{l\ge 1}\binom{t}{2} \nu\big(\hit_{k }(t )\big) p_{\ell}
& = -\sum_{t\ge 0} \binom{t+2}{2} \nu\big(\hit_{k }(t+2 )\big)
%\\&
= -[z^k]\frac{f(z)^2}{\big(1-f(z)\big)^3}
\end{align*}
with $f$ as above. Assembling all this,
\begin{align*}
\mu(\hit_k)&=\mu(A)+\mu(B)+\mu(C)
%\\&
= \nu (\hit_{k } ) +
\binom{k}{2}^{-1}[z^k]\frac{f(z^2)-f(z)^2}{\big(1-f(z)\big)^3}\,.
\end{align*}
Substituting this into~\eqref{mu-Hk} we find
$$
 \pi(\hit_k) = \nu(\hit_k) - \frac{\nu(D\ge 1)}{  1-\nu(D\ge 1)}\Bigg(\binom{k}{2}^{-1}[z^k]\frac{f(z^2)-f(z)^2}
{\big(1-f(z)\big)^3}  +O(k^2/m) \Bigg)\,.
$$
Recalling~\eqref{D1andD2}, we get
$$
 \frac{\nu(D\ge 1)}{  1-\nu(D\ge 1)}=\binom{k}{2}m^{-1}+O(k^4/m^2)
$$
and hence
$$
 \pi(\hit_k)  = \nu(\hit_k) -\big(m^{-1} +O(k^2/m^2)\big)[z^k]\frac{f(z^2)-f(z)^2}{\big(1-f(z)\big)^3}  +O(k^4/m^2) \,.
$$
The constants implicit in the $O(\cdot)$ terms arise in~\eqref{D1andD2} and consequently are functions of $k$ and $m$ alone.  Since $p_0=0$, we have  $P_k=\pi(\hit_k)$ and $R_k=\nu(\hit_k)$, and the theorem follows.
\end{proof}

\begin{proof}[\textbf{\em Proof of Theorem~\ref{thm-renewals-estimates}}]
This follows from Theorem~\ref{thm-renewals-connection}  combined with Lemmas~\ref{l:contour} and~\ref{l:fsquared}, the only condition not explicitly assumed   being condition (i) of Lemma~\ref{l:contour}. This follows from the fact that $\sum p_\ell z^\ell$ converges absolutely for $|z| \le r$ as we assumed in (i) that $p_\ell \leq (1/c) r^{-\ell}$.
\end{proof}

\begin{proof}[\textbf{\em Proof of Theorem~\ref{thm-renewals-size-biased}}]
Conditional upon $\hat{Y}_J=x_J=\ell$ say, the event that $k \in \{\hat{Y}_1,\ldots,\hat{Y}_m\}$ has probability $\pi(\hit_{k-\ell})$ where $\mu$ is defined using the random permutation $\sigma$ of $[m]\setminus \{J\}$. Recall also that $\sum_\ell \ell p_\ell=n/m$. Hence
\begin{equation}\label{HitProb}
\P\left(k \in \{\hat{Y}_1,\ldots,\hat{Y}_m\}\right) = \sum_{\ell\ge 1} \frac{\ell p_\ell}{n/m} \pi(\hit_{k-\ell})\,.
\end{equation}
We can ignore all terms in this  summation with $\ell\ge\log_r (m\log m)$, as by  condition  (i)  they are dominated by the $o(1/m)$ error term in the theorem's claim.
For $\ell<\log_r (m\log m)$, we will apply Theorems~\ref{thm-renewals-connection} and~\ref{thm-renewals-estimates} to the multiset of positive integers $\{x_1,\ldots,x_m\}\setminus \{x_J\}$, noting that  their sum is $n-\ell=n\big(1+O(\log m)/m\big)$, $f$ is replaced by $\hat{f}:=f-z^\ell/m$, and the $p_\ell$ change accordingly, to values we call $\hat{p}_\ell$. Moreover, we use $\hat{r}=r-\epsilon$ in place of $r$, where we choose $\epsilon>0$ such that $\sqrt{r-\epsilon}>1+\delta$, i.e.\ $\epsilon<r-(1+\delta)^2$. Truth of the theorem for such restricted $\epsilon$ implies that  it holds  for all larger $\epsilon$.

We first verify conditions (i--iii) of Theorem~\ref{thm-renewals-estimates} for $\hat f$ etc.\ We have for all $j\ne \ell$ that
\begin{equation}\label{psize}
\hat{p}_j= p_jn/(n-\ell)=p_j\big(1+O(\log m)/m\big)<p_j r/(r-\epsilon)\,,
\end{equation}
and $\hat{p}_\ell \le p_\ell$, which imply that $\sum_{\ell\ge 1} \hat{p}_\ell \hat{r}^\ell \le c $ ($n$ sufficiently large), as required for (i). Note this implies that $\hat f(z)\le c$ for $|z|\le r-\epsilon$, a fact that we will use before long.  We turn next to (iii). For $|z|\le r-\epsilon$, we have $z^\ell/m\le(r-\epsilon)^{(\log m + \log\log m)/\log r}/m =o(1)$, and hence, using the second-last estimate in~\eqref{psize},
$\hat{f}(z)= f(z)+o(1)$.
So, by assumption (ii) of the present theorem and noting that $r-\epsilon>1+\delta$, when $|z|=r-\epsilon$, we have $|\hat{f}(z)-1|=|f(z)-1|+o(1)>\hat{c}_0^{-1} $ where $\hat{c}_0 = 2c_0$ say. This gives what is required for Theorem~\ref{thm-renewals-estimates}(iii)  for $\hat f$ and $\hat r$.
Regarding Theorem~\ref{thm-renewals-estimates}(ii), we first note that $\hat f'(1)=f'(1)+o(1)$ on $|z|\le r-\epsilon$  for much the same reason as $ \hat f=f+o(1)$. Hence (ii) implies for $n$ sufficiently large that $\hat f(z)\ne 1$ when $|z|\le r-\epsilon$ and $|z-1|\ge \delta$. On the other hand, (iii) implies that the mean value   of $f'(z)$   on any path in $|z-1|< \delta$ has absolute value at least $f'(1)/2\ge 1/2$, so the mean value of $\hat f'(z)$ on such a path cannot be zero. Thus, in this disc, the unique solution of $\hat f(z)=1$ is at $z=1$. We now have   Theorem~\ref{thm-renewals-estimates}(ii)  for $\hat f$ and $\hat r$.

Applying Theorems~\ref{thm-renewals-connection} and~\ref{thm-renewals-estimates} in this context, we obtain
\[
\pi(\hit_{k-\ell})=\frac{m-1}{n-\ell} -\big(m^{-1} +O(k^2/m^2)\big) \frac{\hat f'(1)- \hat f'(1)^2+ \hat f''(1)}{\hat f'(1)^3}
+\hat q_k(n)
\]
with
\[
|\hat q_k(n)|\le \hat{c}_0 (r-\epsilon)^{-k}  + O(m^{-1} c_1 (r-\epsilon)^{-k/2})   +O(k^4/m^2)\,.
\]

Again by assumption (ii) of the present theorem, $c_1\le \hat{c}_0^3 \max_{|z|=\sqrt {r-\epsilon}}  | \hat f(z^2)-\hat f(z)^2|$. Recalling that $|\hat f(z)|\le c$ and noting that $c\ge 1$ so $c^2\ge c$, we get  $c_1 \le  c^2\hat{c}_0^3 \le 8c^2{c}_0^3 $, a constant. Since $r-\epsilon>1$ and  $k\to\infty$, and $\hat c_0=2c_0$, we obtain
\begin{equation}\label{qhat}
|\hat q_k(n)| \le \hat c_0 (r-\epsilon)^{-k} + O\left(k^4 / m^2\right) +o(1/m)
\end{equation}
where $o(\cdot)$ depends on the rate that $k\to\infty$.
  Note that $k^2/m^2=o(1/m)$, $ \hat f'(1)= f'(1)+o(1)$, $\hat f''(1) = f''(1)+o(1)$, and $(m-1)/(n-\ell)= (m-1)/n +\ell m/n^2 +o(1/n)$. (Also, $n\ge m$ since each $x_j$ is a positive integer.)  So, recalling that $\sum_{\ell\ge 1}  \ell p_\ell = n/m$,~\eqref{HitProb} gives
\[
\P\left(k \in \{\hat{Y}_1,\ldots,\hat{Y}_m\}\right) =
\frac{m-1}{n}
 -  \frac{f'(1)- f'(1)^2+ f''(1)}{m  f'(1)^3}
+ q_k(n) +  \sum_{\ell\ge 1} \frac{\ell^2 p_\ell m^2}{n^3}\,,
\]
where $ |q_k(n) |$ has an upper bound of the same form as given in~\eqref{qhat}.
Since $\sum_{\ell\ge 1}  \ell^2 p_\ell =f'(1)+f''(1)$ and $f'(1) =n/m$,   cancellation now gives the theorem.
\end{proof}

\begin{proof}[\textbf{\em Proof of Theorem~\ref{thm-renewals}}]
Theorem~\ref{thm-renewals-estimates}(i) holds for $n$ sufficiently large  since $|f(r)- g(r)|<w(n)=o(1)$ and   $g(r)$ is finite by the absolute convergence assumption on $g$. The convergence of $f$ and $g$ in $|z|\le r$ also imply that they are analytic, and so all their derivatives exist, in $|z|< r$. We have $|f'(1)-g'(1)|<w(n)=o(1)$ and hence $g'(1)=n/m+o(1)$. Thus for some $\delta >0$ with $1+\delta +\epsilon'<r$ (we can assume $\epsilon'$ is arbitrarily small, so this is always possible) we have  $|g'(z) -g'(1)|<g'(1)/3 $ if  $|z-1|< \delta$, $z\in \C$.  Arguing as above, this implies  Theorem~\ref{thm-renewals-size-biased}(iii) for $n$ sufficiently large.   Since  $g(z) \ne 1$ for $|z|\le r$ whenever $z\ne 1$, and by the continuity of $g$,  there exists $c_0>0$ such that $|g(z)-1| \ge 2c_0^{-1}$ if $|z|\le r$ and $|z-1|\ge \delta$, $z\in \C$, similarly yielding   both Theorem~\ref{thm-renewals-estimates}(iii)  and   Theorem~\ref{thm-renewals-size-biased}(ii).
 Theorem~\ref{thm-renewals-estimates}(ii) follows from these, as was shown in the proof of  Theorem~\ref{thm-renewals-size-biased}. So the hypotheses of both of these theorems hold. In the conclusions, we note for (a) (the non-size-biased case) that the difference between $\frac{g'(1)- g'(1)^2+ g''(1)}{g'(1)^3 }$ and the corresponding function of $f$ is $o(1)$ by the hypotheses of the corollary and the fact that $g'(1)\ne 0$, which is absorbed in the error term $o(m^{-1})$ in $q_1$.  So Theorem~\ref{thm-renewals-connection} completes the proof of (a). On the other hand, (b) (the size-biased case) follows immediately from Theorem~\ref{thm-renewals-size-biased}.
\end{proof}

%%%%%%%%%%%%%%%%%%%%%%%%%%%%%%%%%%%%%%%%%%%%%%%%%%%%%%%%%%%%%%% SECTION 3
\section{Second moment analysis of cycle factors: framework}\label{sec-second-moment-framework}
In this section we set the framework for a delicate asymptotic analysis of the second moment of the number of
$k$-cycle factors in a random 3-regular multigraph generated by the \emph{configuration model} $\cP_{n,3}$. In this model,
introduced first by Bollob\'{a}s \cite{Bollobas1} (a different flavor of it was implicitly used by Bender and Canfield~\cite{BC}), one associates each of the $n$ vertices with a triplet of distinct points (also referred to as ``half-edges'') and consider a uniform perfect matching (a pairing) on these points. The random $3$-regular multigraph is obtained by contracting each of the triplets into a single vertex (possibly introducing multiple edges and self-loops). Clearly, on the event that this produces a simple graph, it is uniformly distributed among all cubic graphs, and furthermore, this event occurs with probability bounded away from $0$. Hence, every event that occurs w.h.p.\ for this model, also occurs w.h.p.\ for a random $3$-regular graph, and it will suffice to prove our results in this framework. See~\cites{Bollobas2,JLR,Wormald} for additional information.

Let $\cfact$ denote the number of $k$-cycle factors in a configuration of $\cP_{n,3}$.
The bulk of the analysis of $\cfact$ will be carried out in~\S\ref{sec-intersection}, where our goal is to establish that
$ \E[\cfact^2]\leq (3+o(1))\E[\cfact]^2$
provided $k\geq K_0(n)$, with $K_0(n)$ as defined in Eq.~\eqref{eq-K0}. We first estimate the first moment, $\E[\cfact]$, which is trivial by comparison.

Rather than working directly with $\cfact$, %the number of $k$-cycle factors in a pairing $G\sim \cP_{n,3}$,
it will be more convenient to consider a re-scaled version of this variable. A $k$-cycle factor is \emph{ordered} if its cycles are linearly ordered, \emph{rooted} if each cycle is rooted at one of its vertices (i.e., a vertex of each cycle is distinguished as its \emph{root}), and \emph{directed} if each cycle is given one of the two possible orientations. Let
\[ Y_k = \#\{\mbox{\emph{rooted, ordered} and \emph{directed} $k$-cycle factors in $G \sim \cP_{n,3}$}\}\,,\]
and call members of this set {\footnotesize \it{ROD}} \emph{factors}
 for brevity.
 Observe that %each $\rod$ factor is simple (has no loops or multiple edges) and that
 each $k$-cycle factor corresponds to precisely $(n/k)!(2k)^{n/k}$
distinct $\rod$ factors,
 %(via ordering, rooting and directing its cycles)
so that
\begin{equation}
  \label{eq-rod-rescale}
  \frac{Y_k}{\cfact} = \left(\frac{n}k\right)!\,(2k)^{n/k} = (1+o(1))\sqrt{\frac{2\pi n}k}\left(\frac{2n}e\right)^{n/k}\,.
\end{equation}
In the next subsections we establish the first and second moment of $Y_k$, followed by an analysis of $Y_k$ given prescribed sequences of small cycle counts (to be used in the framework of the small subgraph conditioning method).

\subsection{Expected number of cycle factors}
It is straightforward to estimate the expectation of $Y_k$. A given rooted, ordered and directed cycle factor is in one-to-one correspondence with a permutation on the $n$ vertices (we will often use this form), and in order to give rise to it in $\cP_{n,3}$ we choose the incoming and outgoing half-edges of each vertex in its cycle (a total of $6^n$ options) and then add a perfect matching on the remaining half-edges, for which there are
\[\match(n)\deq(n-1)!! =\frac{n!}{(\frac{n}2)!\,2^{n/2}}\]
options. Recalling that $|\cP_{n,3}| = \match(3n)$ we can conclude that
\begin{equation}
   \label{eq-E[Y]}
   \E Y_k = (1+o(1))\frac{\sqrt{2\pi n} \left(\frac{n}{e}\right)^n \cdot 6^n \cdot \sqrt{2} \left(\frac{n}{e}\right)^{n/2}}{\sqrt{2}\left(\frac{3n}{e}\right)^{3n/2}} = (1+o(1))\sqrt{2\pi n} \left(\tfrac43\right)^{n/2}\,.
 \end{equation}
It thus follows from~\eqref{eq-rod-rescale} that
\begin{equation}
  \label{eq-E[F]}
  \E [\cfact] = (1+o(1))\sqrt{k}\bigg(\frac{e(4/3)^{k/2}}{2n}\bigg)^{n/k}\,.
\end{equation}
Observe that whenever $e(4/3)^{k/2}/(2n) \geq 1$, which occurs when $k$ is at least
\[ 2\log_{4/3}(2n/e) = [1-\tfrac12\log_2 3]^{-1}\log_2 (2n/e) \,,\]
we have $\E [\cfact] \geq (1+o(1))\sqrt{k} \to \infty$. On the other hand, if for instance $e(4/3)^{k/2}/(2n) < 1-\frac{k\log k}n$ we
get that $\E [\cfact] \to 0$, which holds for any $k \leq K_0(n) - 2\log_{4/3}[1/(1-\frac{k\log k}n)]$.
As the last additive term is of order $\frac{k\log k}n = o(\frac{\log^2n}n)$, we conclude that $\E[\cfact]=o(1)$ whenever $k \leq K_0(n)-\frac{\log^2 n}n$.

\subsection{Second moment via intersection patterns}
To estimate $\E Y_k^2$ we will count the triples in the set
\begin{equation}
  \label{eq-GRR}
  \left\{ (G, R_1, R_2) ~~:~~ G \in \cP_{n,3}\mbox{ and }R_1,R_2\mbox{ are $\rod$ factors of $G$}\right\}\,,
\end{equation}
where we say that $R$ is a $\rod$ factor of a pairing $G$ if each of its underlying cycles belongs to $G$.
Define the following quantity to capture the potential underlying intersection graph of two factors:
\begin{equation}
  \label{def-intersection-pattern}
  \inter_{h,m} = \left\{ \begin{array}{c}\mbox{undirected graphs on $[n]$ with $h+m$ components:} \\ \mbox{$h$ cycles of length $k$ and $m$ nontrivial paths}\end{array}\right\}
\end{equation}
(where here and in what follows, a path is nontrivial if it has at least 2 vertices). We refer to the graphs --- usually denoted $S$ --- in any $\inter_{h,m}$ as {\em intersection patterns}.

For two $\rod$ factors $R_1,R_2$ and $S \in\inter_{h,m}$ we say $R_1\cap R_2$ \emph{projects} to $S$, denoted by $R_1 \cap R_2 \hookrightarrow S $, if the underlying undirected graph of $R_1 \cap R_2$ is equal to $S$ (that is, the undirected edges in common in $R_1,R_2$ give $S$).

Crucially, $R_1\cap R_2$ necessarily projects to some $S \in \inter_{h,m}$ for some $h,m$. Indeed, let $R_1,R_2$ be two $\rod$ factors. Upon removing the ordering, rooting and directing of the cycles we are left with some number $h$ of $k$-cycles that belong two both factors. On the remaining vertices we have two $k$-cycle factors $H_1$ and $H_2$ of a $3$-regular multigraph on $n-kh$ vertices, whose intersection is a set of simple nontrivial paths with the following property: if we traverse a common path along a $k$-cycle of $H_1$, then as soon as this path ends we see an edge exclusive to $H_1$ followed immediately by a new nontrivial common path (and similarly for $H_2$). This is because $H_2$ forms a spanning 2-regular multigraph in our 3-regular multigraph, and therefore there cannot be two edges exclusive to $H_1$ incident to the same vertex.

Say that a $\rod$ factor {\em contains} $S\in\inter_{h,m}$ if its underlying undirected graph contains $S$, and set
\begin{equation}
  \label{eq-def-N(S)}
  \sN(S) = \#\left\{ \rod\mbox{ factors containing $S$}\right\}
\end{equation}
(for any $h,m$ and $S\in\inter_{h,m}$).
Apart from the orienting of the paths, and the rooting and directing of the cycles, $\sN(S)$ corresponds to the number of different ways one can weave the $m$ paths in $S\in\inter_{h,m}$ into a collection of $n/k - h$ disjoint $k$-cycles by repeatedly connecting the ends of two paths by an edge.

We now claim that the quantity defined in~\eqref{eq-GRR} satisfies
\begin{equation}
  \label{eq-GRR2}
  \#(G, R_1, R_2) = \sum_h \sum_m \sum_{S\in \inter_{h,m}} \sN(S)^2 6^n \match(n-2m)\,.
\end{equation}
Indeed, the factor $\sN(S)^2$ accounts for choosing $R_1,R_2$ given the intersection pattern $S \in \inter_{h,m}$.
Once we compose the factors $R_1$ and $R_2$, the $2m$ endpoints of the $m$ paths in $S$ reach degree 3 (each having two exclusive edges in $R_1,R_2$ and
one common edge in $S$) and all other vertices are of degree $2$. Thus, it remains to complete the pairing by matching the latter $n-2m$
vertices, and finally to order the three half-edges of each vertex (the factor of $6^n$).

Recalling that $|\cP_{n,3}|=\match(3n)$ it then follows from~\eqref{eq-GRR2} that
\begin{equation}
  \label{eq-E[Y^2]}
\E Y_k^2 = \frac{6^n}{\match(3n)} \sum_h \sum_m \match(n-2m) \sum_{S \in \inter_{h,m}} \sN(S)^2\,.
\end{equation}
The lion's share of the work involved in estimating this quantity will be devoted to intersection patterns with $h=0$ (i.e., cycle-free) and $m$ close to $n/3$; these will turn out to account for most of the right-hand side of~\eqref{eq-E[Y^2]}.
The contribution of the remaining pairs $h,m$ will then be shown, partly by reducing to the earlier analysis, to be negligible in comparison with these main terms.

The next theorem is the heart of the matter. Its proof, which relies, {\em inter alia}, on the renewal estimates of~\S\ref{sec-renewals}, is given in~\S\ref{sec-intersection}.

\begin{theorem}
  \label{thm-N(S)}
 Let $m$ and $n$ be such that $|\frac{m}n -\frac{1}3| \leq (\log n)^{-1/3}$, and define
$\sN(S)$ as in~\eqref{eq-def-N(S)}. If $k > (2+\epsilon)\log_2 n$ for some fixed $\epsilon>0$ then
\[ \sum_{S\in\inter_{0,m}} \sN(S)^2  \leq (9+o(1))\left(m!\, 2^m\right)^2\left|\inter_{0,m}\right|\,.\]
\end{theorem}
What we will actually find is that most of this sum comes from those $S\in\inter_{0,m}$ for which $\sN(S) \leq (3+o(1))m! 2^m$.
We may contrast this with, for example, the crude bound
\begin{equation}\label{eq-N(S)-overestimate}
\sN(S)\leq k^{n/k}m!2^m\,.
\end{equation}
Here we think of each member of $\sN(S)$ as corresponding to a permutation of the $m$ paths,
gotten by listing the cycles in their $\rod$ order and then, within each cycle, beginning with the path
containing the root and listing the remaining paths in the order dictated by the orientation of
the cycle.
The factors $2^m$ and $k^{n/k}$ account for orienting the paths and rooting the cycles.
Of course this is wasteful in two ways:  first, we have overpaid for the roots (which are chosen from the first path in a cycle rather than from the entire cycle),
and second,
the only permutations arising from $S$ are those with the property that
all multiples of $k$ are partial sums of their sequence of path sizes.
In particular, consideration of the
probability of the property just mentioned leads naturally to renewals and the
material of \S\ref{sec-renewals}.

%Observe that, trivially,
%\begin{equation}\label{eq-N(S)-overestimate}
%\sN(S) \leq k^{n/k} m!\, 2^m \quad\mbox{ for any $S \in \inter_{0,m}$}
%\end{equation}
%since the number of directed, ordered and \emph{non-rooted} $k$-cycles containing $S$ is clearly at most $m!\,2^m$ (this overestimate ignores the requirement to align the paths so that none of them gets broken off by the $k$-cycles) and the pre-factor $k^{n/k}$ accounts for selecting a root for each cycle.
%Comparing this bound with the one established in Theorem~\ref{thm-N(S)} we see that the contribution of a typical intersection pattern $S \in \inter_{0,m}$ can be estimated by a bound which replaces $k^{n/k}$ by a constant (namely, $3$).  Indeed, many of the permutations of the $m$ paths counted in~\eqref{eq-N(S)-overestimate} do not line up to form $n/k$ oriented $k$-cycles, linking our argument to renewals without replacement.
% The proof of Theorem~\ref{thm-N(S)} will repeatedly examine the probability that a random permutation of the remaining paths produces an additional $k$-cycle using the renewal estimates developed in~\S\ref{sec-renewals}.

\subsection{Properties of a typical intersection pattern}
To begin, we will argue that
\begin{equation}\label{eq:I0m}
|\inter_{0,m}| = \binom{n-m-1}{m-1}\frac{n!}{m!2^m}
\end{equation}
and that the multiset of path sizes of a uniform $S\in \inter_{0,m}$ follows the same law as the multiset of part sizes of a
uniformly chosen member of $\cC_{0,m}$, the set of $m$-compositions of
the integer $n$ with parts of size at least 2.
To see this, regard $T=(a_1,\ldots, a_m)\in \cC_{0,m}$
as a graph in
$\inter_{0,m}$ consisting of the paths $(1,\ldots, b_1),
(b_1+1,\ldots, b_2),\ldots, (b_{m-1}+1,\ldots, b_m)$,
where $b_i= a_1+\cdots +a_i$ (in particular $b_m =n$).
Each relabeling of the vertices of $T$ by $[n]$ gives an $S\in \inter_{0,m}$ whose
path sizes coincide with the part sizes of $T$, and, conversely, this procedure
(choose $T$, relabel its vertices) gives rise to each $S\in \inter_{0,m}$ precisely
$m!2^m$ times.  This gives the second assertion above, and also \eqref{eq:I0m}
once we recall that $|\cC_{0,m}|=\binom{n-m-1}{m-1}$.

The next lemma, which will be used several times in~\S\ref{sec-intersection}, addresses large deviations for the path sizes of a typical intersection pattern $S$, as well as a random subset of $S$.
\begin{lemma}\label{lem:binomial}
Let $|\frac{m}n-\frac13|=o(1)$ and
let $S$ be chosen uniformly from $\inter_{0,m}$. Then
for any $t$ and $\ell$ such that $t \ell = o(n)$, the probability that $S$ contains at least $t$ paths of length at least $\ell$ is at most
\begin{equation}\label{eq:PathsBound}
\left(\frac{4e m}{t(2-o(1))^\ell}\right)^{ t }\,.
\end{equation}
Moreover, if $|\frac{m}n-\frac13| < \epsilon$
 for some $0<\epsilon=\epsilon(n)=o(1)$
and $p_\ell$ is the fraction of $\ell$-vertex paths among the $m$ paths in $S$ for some $\ell=\ell(n)\geq 2$, then
\begin{equation}
  \label{eq-m-dev}
   \P\left( \frac{2-\epsilon}{(2+\epsilon)^\ell} \leq p_\ell \leq  \frac{2+\epsilon}{(2-\epsilon)^\ell}\right) \geq
     1-O\left(\sqrt{m}\ell e^{-\frac18 \epsilon^2 m /(\ell 3^{\ell})}\right)\,.
\end{equation}
\end{lemma}
\begin{proof}
According to the discussion preceding the lemma, it suffices to prove the statements in question for
the distribution of part sizes of a uniformly chosen member of $\cC_{0,m}$.
Here it will be convenient to work with another (standard) reformulation, identifying each
member of $\cC_{0,m}$ with a binary string $W$ of length $n-m-1$ and Hamming weight $m-1$;
namely such a word with 1's in positions $a_1,\ldots, a_{m-1}$ corresponds to the composition
$(a_1+1, a_2-a_1+1,\ldots, a_{m-1}-a_{m-2}+1,n-m-a_{m-1}+1)$.
In particular, a maximal interval of zeros of length $\ell$ corresponds to a part of size $\ell+2$ in
the composition. % (or a path of size $l+2$ in each associated member of $\inter_{0,m}$).

We will first establish~\eqref{eq:PathsBound}.
Let $Z^*_\ell$ be the number of parts of size at least $\ell$; equivalently, $Z^*_\ell=\sum_{i=0}^{n-m-\ell+1} I^*_i$ where
 \[ I^*_i= \one_{\left\{ W_i=1\;\text{ and }\;
   W_{i+1}=\ldots=W_{i+\ell-2}=0\right\}}\quad(i=0,\ldots,n-m-\ell+1)\,.\]
We can only have $I^*_i=I^*_j=1$ if
$|i-j|\geq \ell -1$ (i.e., if the sets $\{i,\ldots,i+\ell-2\}$ and $\{j,\ldots,j+\ell-2\}$ are disjoint), and it is easy to see that
\begin{compactenum}[(a)]
\item there are
$\binom{n-m-1-t(\ell-2)}{t}$ possible $t$-subsets $T\subset[n-m-\ell+1]$ satisfying $|i-j|\geq\ell-1$ for all $i\neq j\in T$,
and for each one there are $\binom{n-m-1-t(\ell-1)}{m-1-t}$ strings with Hamming weight $m-1$ such that $I^*_i=1$ for all $i\in T$;
\item there are $\binom{n-m-1-t(\ell-2)}{t-1}$ such $t$-subsets $T\subset\{0,\ldots,n-m-\ell+1\}$ that include $0$, and for each one
there are $\binom{n-m-t(\ell-1)}{m-t}$ strings as above.
\end{compactenum}
Thus,
\begin{align*}
\P(Z^*_\ell \geq t) &\leq \frac{\binom{n-m-1-t(\ell-2)}{t}\binom{n-m-1-t(\ell-1)}{m-1-t}
+ \binom{n-m-1-t(\ell-2)}{t-1}\binom{n-m-t(\ell-1)}{m-t} } { \binom{n-m-1}{m-1} }
\\ &= \left[1 + \frac{t}{m-t}\right] \frac{1}{t!} \frac{(m-1)_{t} \, (n-2m)_{t(\ell-2)}
}{ (n-m-1)_{t(\ell-2)}} \,.
\end{align*}
Since $t \ell =o(n)$ by hypothesis and $m \sim n/3$, we see that $\frac{t}{m-t}=o(1)$ and that each of the factors in the numerator
of the last fraction is $(1+o(1))m$ and each one in its denominator is $(2-o(1))m$. Using $t! \geq (t/e)^t$, this gives
\[ \P(Z^*_\ell \geq t) \leq \bigg(\frac{ e m}{t\left(2-o(1)\right)^{\ell-2}}\bigg)^t \,,\]
which is~\eqref{eq:PathsBound}.

It remains to establish~\eqref{eq-m-dev}. To this end, as before we can move to the setting where $\tW$ is a binary string of length $n-m-1$. In this case however we take each coordinate as an independent Bernoulli($\frac{m-1}{n-m-1}$) variable. We will account for $\P(\sum_i \tW_i = m-1) \gtrsim 1/\sqrt{m}$ later. For notational convenience, extend $\tW$ by defining $\tW_0 = \tW_{n-m} = 1$.

Let $\tZ_\ell$ ($\ell\geq 2$) count the number of parts of size $\ell$ in
the composition corresponding to $\tW$.
We can write $\tZ_\ell$ as a sum of indicators, $\tZ_\ell = \sum_{i=0}^{n-m-\ell} \tI_i$, where $\tI_i = \tI_i(\ell)$ is the event that
coordinates $i,i+1,\ldots,i+\ell-1$ form a part of size exactly $\ell$; that is,
 \[ \tI_i= \left\{ \begin{array}
   {cl}    \tW_i=\tW_{i+\ell-1}=1\,,\\
   \tW_{i+1}=\ldots=\tW_{i+\ell-2}=0
 \end{array}\right\}\quad(i=0,\ldots,n-m-\ell+1)\,.\]
Then, for any $i \notin \{0,n-m-\ell+1\}$,
\begin{equation}
  \label{eq-P(Ii)}
  \P(\tI_i=1) =  \left(\frac12 + O\left(|\tfrac{m}n-\tfrac13|\right) +O(1/n)\right)^{\ell} = \left(\tfrac12+o(1)\right)^\ell \,.
\end{equation}
Similarly the events $\tI_0,\tI_{n-m-\ell+1}$ have a probability of $(\frac12+o(1))^{\ell-1}$ each.
The events $\{\tI_i : i\in A\}$ are clearly mutually independent if $A$ is a set of indices
whose pairwise distances are all at least $\ell$. We can thus partition $\tZ_\ell$ into
$\tZ_\ell = \sum_{j=0}^{\ell-1} \tZ_\ell^{(j)}$ where
\[\tZ_\ell^{(j)} \deq \sum_{i  \equiv j (\bmod{\ell})}\tI_i\] satisfies the following stochastic domination relations for large enough $n$:
\[ \Bin\Big(\big\lfloor\tfrac{n-m-\ell}\ell\big\rfloor,  (2 + \epsilon)^{-\ell}\Big) \preccurlyeq \tZ_\ell^{(j)} \preccurlyeq \Bin\Big(\big\lceil \tfrac{n-m-\ell}\ell \big\rceil, (2- \epsilon)^{-\ell}\Big) + 2\,.\]
(The additive $2$ on the right-hand side accounted for the events $\tI_0,\tI_{n-m-\ell}$.)
The binomial variables on the left and right have means
$(2-o(1))\frac{m}\ell(2+\epsilon)^{-\ell}$ and $(2+o(1))\frac{m}\ell(2-\epsilon)^{-\ell}$ respectively. Thus,
\[\P \left(\frac{m}\ell \frac{2-\epsilon}{(2+\epsilon)^{\ell}} < \tZ_\ell^{(j)} <  \frac{m}\ell \frac{2+\epsilon}{(2-\epsilon)^{\ell}} \right) \geq 1-O\left(e^{-(1-o(1))\frac16 \epsilon^2 \frac{m}{\ell} 3^{-\ell}}\right) \geq 1-O\left(e^{-\frac18 \epsilon^2 m / (\ell 3^\ell)}\right)\,,\]
where the first inequality follows from standard large deviation estimates for the binomial (see, e.g.,~\cite{JLR}*{Corollary~2.3}). A union bound over $j=0,\ldots,\ell-1$ now completes the proof.
\end{proof}

As a corollary, we have the following rough bound, which for $m\sim n/3$ improves on the trivial estimate $\sN(S)\le k^{n/k}m!2^m$ from~\eqref{eq-N(S)-overestimate} for all but an ultimately negligible proportion of the intersection patterns.
\begin{lemma}\label{l:better}
Let $E = \{ S\in\inter_{0,m} \;:\; \sN(S)>(\log k)^{3n/k}m!2^m \}$ for $m$ satisfying $m=(\frac13 +o(1))n$. Then
\[ \sum_{S\in E}\sN(S)^2 = o\left((m!2^m)^2|\inter_{0,m}|\right)\,.\]
\end{lemma}
\begin{proof}
Apply the first part of Lemma~\ref{lem:binomial} with
\[
t=\frac{n}k \frac{\log \log k}{\log k}\,,\quad \ell=\log^{2} k\,\]
(noting that $t\ell = o(n)$ as needed in that lemma). as the base of the exponent in~\eqref{eq:PathsBound} for this choice of $t$ and $\ell$ is
$\left(2-o(1)\right)^{-\log^2 k}$, the probability of at least $t$ paths of length at least $\ell$ is at most
\[
\left(2-o(1)\right)^{-(n/k)\log k \log \log k}
< k^{-5 n/k}\,
\]
for any sufficiently large $k$.
Revisiting~\eqref{eq-N(S)-overestimate}, such path partitions contribute $o((m!2^m)^2|\inter_{0,m}|)$ to the summation of $\sN(S)^2$.
On the other hand, the partitions $S$ with at most $t$ paths of length at least $\ell$ have less than $k^t
\ell^{n/k}$ ways to choose the root vertices from the initial paths of each cycle, and hence \[ \sN(S)\le k^t \ell^{n/k}m!2^m = e^{(n/k)( \log \log k+2\log \log k)}m!2^m\,,\]
and so no such partitions are in $E$. The result now follows.
\end{proof}

\section{Cycle-free intersection patterns and renewals}\label{sec-intersection}

\subsection{Normal and abnormal intersection patterns}\label{subsec:normal-abnormal-def}

At its most na\"ive level, the proof of Theorem~\ref{thm-N(S)} would like to estimate (really meaning bound) $\sN(S)$ for a given $S$ by randomly
ordering the paths and estimating --- via the results of~\S\ref{sec-renewals} --- the probability that all multiples of $k$ are partial sums of the
resulting sequence of path sizes, as well as bounding the number of corresponding $\rod$ factors. (Each ordering with this property --- already mentioned following
Theorem~\ref{thm-N(S)} --- gives rise to a number, crudely between $2^{|S|}$ and $2^{|S|}k^{n/k}$, of $\rod$ factors via directing paths
and rooting cycles, while other orderings do not contribute to $\sN(S)$. However, such bounds are too rough for our ultimate goal.)

The most obvious problem with this is that the statistics ($p_\ell$) of $S$ itself, or of some of the subsets of $S$ remaining after some
paths have been chosen (``suffixes''), might not support the use of Theorem~\ref{thm-renewals}.
In extremely rough terms, one may say that much of the work below is aimed at reducing to situations where these results do apply (though even in the ideal situation where all such suffixes do behave nicely, treated in Theorem~\ref{thm:goodSuffixes} below,
the analysis is more delicate than the preceding description suggests).

We will use --- with various values of $\delta$ --- the following concrete conditions
supporting use of our renewal estimates.

\begin{definition}[$\delta$-normal intersection patterns]
  \label{def:delta-normal}
  Let $S$ be a set of $m$ paths and let $x_1,\ldots,x_m$ denote the lengths (numbers of vertices) of these paths.
The \emph{path-distribution} $(p_\ell)_{\ell\ge 1}$ of $S$ gives the relative frequencies of the $x_i$'s, namely $p_\ell = \frac1m |\{j:x_j=\ell\}|$ for $\ell>0$.
 Given   $0<\delta<1$,  we say that $S$ is $\delta$-\emph{normal}
 if the following conditions hold:
 \begin{enumerate}[(i)]
\item Short paths: $p_1=0$ and for all $\ell \leq \sM := \frac18 \log \log k$,
\begin{align}
  \label{eq-normal-short}\left|p_\ell-2^{1-\ell}\right| \leq \epsilon_\ell = \epsilon_\ell(\delta) :=
  \left(\ell^4 (2-\delta)^\ell\log^{1/8} k\right)^{-1}\,.
  \end{align}
(The lower bound on $p_\ell$ is meaningful since $\epsilon_\ell \leq e^{-\ell}$ for any $\ell \leq \sM$, recalling $k\to\infty$.)
\item Long paths: for all $\ell \geq \sM$,
\begin{align}
  \label{eq-normal-long}
p_\ell \leq \gamma_\ell = \gamma_\ell(\delta) := \left(\ell^4 (2-\delta)^\ell\right)^{-1}\,.
\end{align}
 \end{enumerate}
\end{definition}

\begin{remark}
  \label{rem:delta-normal-cons}
For every $\delta$-normal $S\in\inter_{0,m}$,
there are no paths of length $\ell \geq \log_{2-\delta} n$ (as $k\le n$ so such paths are long, and then $\gamma_\ell \ll 1/n \leq 1/m$), and the following hold:
\begin{align}
&m = \left(\tfrac13 + o(1)\right) n\,,\label{eq-dnorm-sum-paths}\\
&\sum_{\ell=2}^{\sM} \ell^2 (2-\delta)^\ell \epsilon_\ell =o(1)\,,\label{eq-dnorm-short}\\
& \sum_{\ell\geq\sM}      \ell ^2  (2-\delta)^\ell p_\ell =o(1)\,. \label{eq-dnorm-long}
\end{align}
(Condition~\eqref{eq-dnorm-sum-paths}, in place of the explicit bound in Theorem~\ref{thm-N(S)}, suffices for much of what we do.)
\end{remark}

In line with our discussion above, we will be able to estimate the number of $\rod$ factors arising from any $\delta$-normal intersection pattern quite accurately, as the following theorem states.
\begin{theorem}\label{thm:normal}
For every $\epsilon>0$ there exists some $0<\delta<1$ such that the following holds: if $k \geq (2+\epsilon)\log_2 n$ and $S\in\inter_{0,m}$ is $\delta$-normal %and contains at most $n/(k\log\log k)$ paths of length at least $\log_{2-\delta}k$
then
\[ \sN(S) \leq (3+o(1))m!2^m\,.
\]
\end{theorem}

Our plan for the remainder of \S\ref{sec-intersection} is as follows.
We begin here with the aforementioned Theorem~\ref{thm:goodSuffixes}, which deals with
situations in which we never encounter an abnormal path distribution;
it is in the proof of this result, which will
be central to our argument, that the results of \S\ref{sec-renewals} will come into play.
We then, in \S\ref{subsec:N(S)-via-normal}, derive Theorem~\ref{thm-N(S)} from Theorem~\ref{thm:normal}.
The proof of Theorem~\ref{thm:normal} itself, the trickiest
part of the whole business, is given in \S\ref{subsec:bad-suffix}.

For $t\in[n/k]$, the {\em $t$-suffix} of a $\rod$ factor is the sequence of paths in its last $t$ cycles.
(To define ``sequence'' we may regard the \emph{first} path in a cycle as the one containing the root, but actually in what follows we will really only
be interested in the \emph{set} of paths in a suffix.)

Let $\sN_\delta(S)$ be the number of $\rod$ factors containing $S$ in which each of the $n/k$ suffixes has a $\delta$-normal path distribution.
(In particular, $\sN_\delta(S)=0$ unless $S$ is $\delta$-normal.)
The following theorem, central to our proof, provides an estimate on $\sN_\delta(S)$ using the estimates on renewal processes without replacement given in~\S\ref{sec-renewals}.

\begin{theorem}\label{thm:goodSuffixes}
%Fix $0<\delta<\frac18$, let $k \geq 4 \log_2 n$. For any $S\in\inter_{0,m}$ we have
For every $\epsilon>0$ there exists some $0<\delta<1$ such that the following holds: if $k \geq (1+\epsilon)\log_2 n$ then for any $S\in\inter_{0,m}$,
\[ \sN_\delta(S) \leq (3+o(1))m!2^m\,.\]
\end{theorem}

The proof of this (somewhat paradoxically) requires bounding the probability of $\delta$-abnormal suffixes. For this we use the next two assertions concerning $\delta'$-normality of subsets of a $\delta$-normal intersection pattern.  The first of these addresses
random (uniform) subsets.

\begin{claim}\label{clm:normal-abnormal}
Fix $0<\delta< \delta'<1$.
 Let $S'$ be a uniform $m'$-subset of a $\delta$-normal set of paths $S\in\inter_{0,m}$. The following hold.
\begin{compactenum}
  [(1)] \item\label{it-abnorm-m'-m} If $m' \geq m/\log\log k$ then \[\P(\mbox{$S'$ is $\delta'$-abnormal}) \leq e^{-m k^{-o(1)}}\,.\]
  \item \label{it-abnorm-log-k} If $S$ has no path of length at least $\log_{2-\delta}k$ and we let $\Sigma'$ denote the number of vertices in $S'$ and $\theta_0= 1 -\log_{2-\delta}(2-\delta')$ then for $n'\le n$,
       \[\P(\mbox{$S'$ is $\delta'$-abnormal}, \Sigma'= n')\leq \sqrt{n'}e^{- n' k^{-1+\theta_0-o(1)}}\,.\]
  \end{compactenum}
\end{claim}
\begin{proof}
Throughout this proof, write $p_\ell$ for the relative frequency of $\ell$-vertex paths in $S$ and let $p'_\ell$ be the analogous quantity for $S'$.

Consider some $\ell \leq \sM$, and note that for such $\ell$ we have $\left| p_\ell - 2^{1-\ell} \right| < \epsilon_\ell(\delta)$ by the $\delta$-normality of $S$.
As the number of $\ell$-paths in $S'$ is hypergeometric with mean $m' p_\ell$, Hoeffding's inequality for hypergeometric variables~\cite{Hoeffding}*{Theorem~2 and~\S6} implies that for any $\alpha>0$,
\[ \P\left(\left|p'_\ell - p_\ell\right| \ge \alpha\right) \leq 2\exp(- 2\alpha^2 m'  )\,. \]
Substituting $\alpha=\epsilon_\ell(\delta')/\log k$, for instance, we get that with probability $1-\exp(-m' k^{-o(1)})$ we have $ |p'_\ell - p_\ell | = o(\epsilon_\ell(\delta'))$.
In this case, combining the facts $\left| p_\ell - 2^{1-\ell} \right| < \epsilon_\ell(\delta)$ and $ \epsilon_\ell(\delta')/ \epsilon_\ell(\delta) = \big(1+\frac{\delta'-\delta}{2-\delta'}\big)^\ell > 1$
 implies (recalling $\delta$ and $\delta'$ are fixed) that $|p'_\ell-2^{1-\ell}|<\epsilon_\ell(\delta')$. A union
 bound now establishes the short-path condition for all $\ell\leq\sM$ with the same probability guarantee of $1-\exp(-m' k^{-o(1)})$ (since we can assume $m'\ge 1$).

Part~\eqref{it-abnorm-m'-m} of the claim now easily follows, since the long-path condition is fulfilled deterministically for $S'$, as the following argument shows.
Let $\ell\geq \sM$. By hypothesis, there are at most $m \gamma_\ell(\delta)$ paths of length $\ell$ in $S$, and thus also in $S'$, yielding that
 \[ p'_\ell \leq (m/m')p_\ell \leq \gamma_\ell(\delta) \log\log k\]
using our hypothesis on $m'$ in this part. However, $ \gamma_\ell(\delta')/\gamma_\ell(\delta) = %\left(\frac{2-\delta}{2-\delta'}\right)^{\ell} =
\big(1 + \frac{\delta' - \delta}{2-\delta'}\big)^{\ell}$,
which is at least $(\log k)^c$ for some $c=c(\delta,\delta')>0$. For large enough $n$, this outweighs the $\log\log k$ factor from above and gives
$ p'_\ell \leq \gamma_\ell(\delta')$
for all $\ell\geq \sM$.

It remains to prove Part~\eqref{it-abnorm-log-k}. First we will show that $\Sigma' \le 4m'$ with probability $1-\exp(- n' k^{-o(1)})$.
Recall that we can assume that $S$ is $\delta$-normal, and hence $m=m/3+o(n)$ by~\eqref{eq-dnorm-sum-paths}. If $m'$ differs from $n' m/n \sim n'/3$ by more than $ n'/100 $ (say) then, using the bound on the maximal length of a path, Hoeffding's inequality for the concentration of hypergeometric variables~\cite{Hoeffding} again shows that for some absolute constant $c>0$
\[
\P(\Sigma' = n') \leq \exp\bigg(-c \frac{ (n')^2}{m' (\log_{2-\delta} k)^2}\bigg) \leq
\exp\left(- n' k^{-o(1)}\right)\,,
\]
using the fact that $m'\leq n'/2$ (otherwise this probability is $0$ since every path in $S$ has length at least 2).
Comparing this probability with the desired estimate in the statement of the claim, we may now assume $m' \geq n'/4$ (as $\theta_0<1$ is fixed).

It will be convenient to approximate $S'$ via $\tilde{S}$ that contains each path from $S$ independently with probability $m'/m$.
As in the proof of Lemma~\ref{lem:binomial}, since $\P(|\tilde{S}|=m') \gtrsim \frac1{\sqrt{m'}} \gtrsim \frac1{\sqrt{n'}}$  and on this event $\tilde{S}$ has the same distribution as $S'$,
one can infer properties of $S'$ from those of $\tilde{S}$ via a multiplicative cost of $O(\sqrt{n'})$ in the probability.

The number $\ell$-vertex paths in $\tilde{S}$ is simply a $\Bin(m p_\ell, m'/m)$ variable,
where $p_\ell$ is the relative frequency of $\ell$-vertex paths in $S$.
Consider $\sM \leq \ell \leq \log_{2-\delta}k $ (with the upper bound justified by our hypothesis in this part). The number of $\ell$-vertex paths in $S'$ would violate $\delta'$-normality only if it should exceed
\[ m' \gamma_\ell(\delta') = m' \gamma_\ell(\delta) \bigg(\frac{2-\delta}{2-\delta'}\bigg)^\ell\,.\]
As such, large deviation estimates for the binomial distribution, %(see, e.g.,~\cite{JLR}*{Corollary~2.3}),
%(see, e.g.,~\cite{Hoeffding}*{Theorem 3, inequality~(2.9)} and Bennet~\cite{Bennett}*{inequality~(8b)})
%that if $X$ is binomial and $\delta >0$ then $\P\left(X > a+\E X\right)\leq \exp\left( -\phi(\frac{a}{\E X}) \E X\right)$ where $\phi(x) = (1+x)\log(1+x)-x$.
together with the fact that $p_\ell \leq \gamma_\ell$ thanks to the $\delta$-normality of $S$,
show that the probability of this event is at most
\[\exp\bigg(-m'\gamma_\ell(\delta) \bigg(\frac{2-\delta}{2-\delta'}\bigg)^{\ell+o(\ell)}\bigg)
= \exp\big(-m'(2-\delta') ^{-(1+o(1))\ell}\big)
\]
(the $o(\ell)$ term is an additive term of order $\log \ell$ working in our favor, but that will not be needed).
Since $\ell \leq \log_{2-\delta} k$ we have
$(2-\delta') ^{\ell}= (2-\delta) ^{(1-\theta_0)\ell} \leq k^{1-\theta_0}$,
and now a union bound over $\sM\leq\ell \leq \log_{2-\delta} k$, including the factor $ \sqrt{n'} $ from above, completes the proof.
\end{proof}
The second of the above-mentioned supporting assertions for proving Theorem~\ref{thm:goodSuffixes} says that for $\delta'>\delta$
and sufficiently large $n$,
any large enough subset of a $\delta$-normal path distribution
is (deterministically) $\delta'$-normal.
\begin{claim}
  \label{clm:normal-large-subset}
Fix $0<\delta<\delta'<1$
 and $\alpha>0$. If $S\in\inter_{0,m}$ is $\delta$-normal and $S'$ is an $m'$-subset of $S$ for $m'\geq m - n k^{-\alpha}$ then $S'$ is $\delta'$-normal for large enough $n$.
\end{claim}
\begin{proof}
Let $p_\ell$ and $p'_\ell$ be the relative frequencies of $\ell$-vertex paths in $S$ and $S'$, resp.

First consider the long-path condition. By the $\delta$-normality of $S$, for any $\ell\geq \sM$ we have at most
% $m \gamma_\ell(\delta) \sim m' \gamma_\ell(\delta)$ paths of length $\ell$ in $S$. Thus,
$m \gamma_\ell(\delta) $ paths of length $\ell$ in $S$. Thus, using~\eqref{eq-dnorm-sum-paths} to see that $m\sim m'$,
\[ p'_\ell \leq (m/m')\gamma_\ell(\delta) = (1+o(1))\gamma_\ell(\delta)
<\gamma_\ell(\delta')\,,\]
with the last inequality following from the fact, already used in the proof of Claim~\ref{clm:normal-abnormal}, that $\gamma_\ell(\delta')/\gamma_\ell(\delta) = %\left(\frac{2-\delta'}{2-\delta}\right)^{\ell} =
\big(1 + \frac{\delta' - \delta}{2-\delta'}\big)^{\ell}\to\infty$ with $\ell$ (as a poly-log).

Now take $\ell \leq \sM$. The $\delta$-normality of $S$ implies that
$\left| p_\ell - 2^{1-\ell} \right| < \epsilon_\ell(\delta)$.
Observe that if $d_\ell \leq n k^{-\alpha}$ is the number of $\ell$-vertex paths in $S\setminus S'$ then
\[ \left|  p'_\ell - p_\ell \right| =\frac{1}{m'} \left| m p_\ell (1-m'/m) - d_\ell\right| \leq \frac{m-m'}{m'} + \frac{d_\ell}{m'} = O(k^{-\alpha})\,.\]
This term is therefore negligible compared to $\epsilon_\ell(\delta) \geq k^{-o(1)}$, and so
\[ |p'_\ell - 2^{1-\ell}| \leq (1+o(1))\epsilon_\ell(\delta)\,.\]
The proof is now concluded by the fact $ \epsilon_\ell(\delta')/ \epsilon_\ell(\delta) = \big(1+\frac{\delta'-\delta}{2-\delta'}\big)^\ell > 1$.
\end{proof}

With the above two claims at our disposal, we turn to Theorem~\ref{thm:goodSuffixes}.

\begin{proof}[\textbf{\emph{Proof of Theorem~\ref{thm:goodSuffixes}}}]
%By definition, $\sN(S)$ counts the number of ways to partition $S$ into subsets $\{S_j\}$, each with precisely $k$ vertices (corresponding to a $k$-cycle), then order, direct and root the cycles. The extra constraint imposed in $\sN_\delta(S)$ is that for any $l$ the path-distributions of $\cup_{j\geq l} S_j$ would be $\delta$-normal.
%
Fix $\epsilon_0 > 0$ arbitrarily small.
For sufficiently small $\delta$, the upper bound on $\sN_\delta(S)$ will be established by induction: we will show that if $S'$ has $m'$ paths and a total of $n' = r k$ vertices then deterministically
\[ \sN_\delta(S') \leq (3+\epsilon_0) e^{\sum_{t=2}^{r}\frac{\psi(t)}{k t}} \, (m')! 2^{m'}\,,\]
for some function
\[ \psi(t) = \psi(k,t) = \left\{\begin{array}
  {ll}
  O(k^{3/4})   & t  < k^4\\
   o(1) & t \geq k^4\,.
 \end{array}\right.
\]
This will of course imply the desired statement on $\sN_\delta(S)$ (recalling that $\epsilon_0$ can be taken arbitrarily small) since
%$r=n/k$ would then correspond to an error pre-factor of
then for $r=n/k$ the exponential pre-factor will be
\begin{align*}
1 + O\bigg(\sum_{t<k^4} \frac{k^{3/4}}{kt} \bigg) + o\bigg(\sum_{t= k^4}^{n/k} \frac{1}{kt}\bigg) = 1 + k^{-1/4+o(1)} + o\bigg(\frac{\log n}{k}\bigg) = 1+o(1)\,,
\end{align*}
where the last inequality used the fact that $k$ has order at least $\log n$.

We may (and do) assume that $S'$ is $\delta$-normal, since otherwise $\sN_\delta(S')=0$ by definition. Hence,~\eqref{eq-dnorm-sum-paths} implies that $n' \sim 3 m'$, a fact we will use several times.
For $r=1$ clearly there are $(m'-1)! 2^{m'}$ possibilities for ordering and directing the $m'$ paths into an ordered and directed cycle, which then offers $n'=k$ possible roots to form a $\rod$ factor.
Since  $k=n' \sim 3m'$, indeed $ \sN_\delta(S) \leq (3+\epsilon_0) (m')! 2^{m'}$ for $n$ sufficiently large.

Next, consider $n' = r k$ for some $r > 1$.
A natural way to carry out the induction would be to count the ways to select some ordered subset $A\subset S'$ of paths which together sum up to $k$ vertices, direct each of them, distinguish one of the $k$ points as a root and proceed to select the remaining cycles recursively.
%While rooting the first cycle via this approach is immediate,
%estimating the probability that some $j$ random paths sum up to $k$ and at the same time multiplying by a factor of $1/j$ (accounting for the different ways of choosing the same ordered cycle prior to rooting it) calls for a delicate joint estimate of typical values of $j$ and their corresponding probabilities of hitting $k$.
We must then divide by $j$, being the number of ordered subsets that produce the same cyclic order prior to rooting. While the part involving the root is easy, it turns out that estimating the probability that some $j$ random paths sum up to $k$, divided by $j$, would call for a delicate joint estimate of typical values of $j$ and their corresponding probabilities of hitting $k$, and it is unclear how to estimate these probabilities to the desired accuracy.

Instead, here we will compose the first cycle out of some $j$ paths differently. We first select its root (for which there are $n'$ possible candidates) corresponding to some path in $S'$, complement it with a selection of some $j-1$ paths in an ordered manner, multiply by the $2^j$ ways of directing the paths and finally weigh in the probability that the paths thus chosen contain $k$ vertices. This corresponds to the renewal problem studied in \S\ref{sec-renewals}: we choose an element of $S'$ in a size-biased sample (the root path) and permute the remaining elements uniformly, then ask for the probability of hitting $k$ via one of the partial sums.
Let $\hit_k$ denote this event, let $\hit_k^{(j)}$ denote the event of having the first $j$ elements sum to $k$ and, for a given subset $A$, write $\hit_k(A)$ for the event that its elements sum up to $k$.
The astute reader will notice that we have just introduced notation for the very event whose probability we mentioned above is hard to estimate. Next we will proceed to wave our magic wand and make $\hit_k^{(j)}$ disappear in a puff of summation signs.

By slight abuse of notation we write $v\in A$ for a vertex $v$ to denote that $A$ contains a path going through $v$.
%, and with this notation
Immediately from the definition, any $(r-1)$-suffix of a $\rod$ factor counted in $\sN_\delta(S')$ is also counted in $\sN_\delta(S'\setminus A)$. Hence,
we can apply the induction hypothesis to get that
\begin{align*}
&\sN_\delta(S')= \sum_{j}\sum_{v\in[n']}\sum_{\substack{A \subset S'\\ |A| = j,\,v\in A}}
\!\!\! (j-1)! 2^j \one_{\hit_k(A)} \sN_\delta(S' \setminus A) \\
&\leq \sum_{j}(j-1)! 2^j \!\sum_{v\in[n']} \!\sum_{\substack{A \subset S'\\ |A| = j,\, v\in A}}
\!\!\! \one_{\hit_k(A)} (3+\epsilon)e^{\sum_{t=2}^{r-1}\frac{\psi(t)}{kt}}  (m'-j)! 2^{m'-j} \\
&\leq (3+\epsilon)e^{\sum_{t=2}^{r-1}\frac{\psi(t)}{kt}} n' (m'-1)! 2^{m'} \P(\hit_k)\,,
\end{align*}
where the last inequality is by the fact that the number of choices for a distinguished vertex $v$, followed by $j-1$ \emph{ordered} elements of $S'$ together summing to $k$, is precisely $n'(m'-1)\cdots(m'-j+1)\P\big(\hit_k^{(j)}\big)$, and so we can collect the factorial terms independently of $j$ and remain with $\sum_j \P\big(\hit_k^{(j)}\big) = \P(\hit_k)$.

We now wish to show that $\P(\hit_k) \leq \left(1+\frac{\psi(r)}{n'}\right) \frac{m'}{n'} $.
If $r \geq k^4$ we appeal to the estimate in Theorem~\ref{thm-renewals}(b) for the size-biased renewal process without replacement. Let $g(z)=\sum_{\ell\ge 2} 2^{1-\ell} z^\ell$ and let $R=2-\delta$. Then $g(z)$ is absolutely convergent for $|z|\le R$, where its value is $z^2/(2-z)$, and hence the unique solution of $g(z)=1$ in  $|z|\le R$ is at $z=1$. Also let   $f(z) =\sum_{\ell\ge 2} p_{-\ell} z^\ell$. Since $S'$ is  $\delta$-normal we have $|f(z)- g(z)|+ |f'(z)- g'(z)|+|f''(z)- g''(z)|<w(n)$  for $|z|\le R$, where $w(n)= o(1)$ by~\eqref{eq-dnorm-short} and~\eqref{eq-dnorm-long} and $w$ depends only on $\delta$ and $k$. Thus, $f$ and $g$ satisfy the requirements of Theorem~\ref{thm-renewals}, and so by (b) we have for any $\epsilon'>0$
\[
|Q_k-m'/n'|\le  o(1/m') + O\big((R-\epsilon')^{-k}+ k^4 / (m')^2\big)\,.
\]
Recalling~\eqref{eq-dnorm-sum-paths}, $m'\asymp n'$, and since  $k^4\le r=n/k$ we have  $k^4 /  (m')^2=o(1/n')$. To make the other error term $o((1/n')$  as well, since $k > (1+\epsilon)\log_2 n$ we can simply choose $\delta=\delta(\epsilon)$  sufficiently small such that $( 1+\epsilon)\log_2(2-\delta) > 1$, and then choose $\epsilon'$ to be sufficiently small enough such that
$ (R-\epsilon')^{-k}< (2-\delta-\epsilon')^{(1+\epsilon) \log _2 n}=o(1/n')$.
 That is, we have $\psi(r)=o(1)$ in this case.

When $2 \leq r \leq k^4$ (equivalently, $n' \leq k^5$) we  modify this strategy to bypass the error term $O(k^4(n')^{-2})$ (which is large enough to foil the entire framework already when $n' \asymp  k^2$).
Again we choose the first path in a size-biased way, so as to determine the root of the next cycle, but now we condition on this path. Denoting its length by $\ell_0$, observe that $\ell_0\leq \log_{2-\delta}n' \leq 5 \log_{2-\delta} k$ by the assumption that $n' = rk \leq k^5$.

Note that the remaining paths have a deterministically $\delta_0$-normal path distribution for any $\delta_0>\delta$ and large enough $n$ thanks to Claim~\ref{clm:normal-large-subset} (as we move from a $\delta$-normal set of $m'$ paths to an $(m'-1)$-subset of it).

%We claim that the remaining paths are deterministically $\delta_0$-normal for any $\delta_0>\delta$ and large enough $k$. Indeed,
%the long-path condition is satisfied since the number of paths gets multiplied by a factor of $(1-\frac{1}{m'-1})$, easily accommodated by the change from $\delta$ to $\delta_0$ in $\gamma_\ell$ (which amounts to a factor of $(\log k)^{c(\delta,\delta_0)}$ for some fixed $c>0$). The short-path condition is again fulfilled since $2^{1-\ell} \frac{1}{m'-1} = o(\epsilon_\ell)$ thanks to the fact that $m' = (3-o(1))n' \geq (3-o(1))k$ whereas both $2^{1-\ell}$ and $\epsilon_\ell$ are at least $k^{-o(1)}$. (Note that removing the first path has no bearing on the short-path condition since $m \epsilon(\delta) \geq k^{1-o(1)} \gg 1$ for any $\ell\leq\sM$.)
%As before $\epsilon_\ell(\delta_0) < \epsilon_\ell(\delta)$ for any $\delta_0>\delta$ and altogether we have $\delta_0$-normality.

Subsequently, we choose paths uniformly (without replacement), but this time we do so until reaching (or exceeding) a total of $y = \lfloor k - 10\log_{2-\delta} k\rfloor$ vertices in these paths (including the first one). Let $m_0$ and $n_0$ denote the numbers of remaining paths and vertices at that point, and write
 \[ \Delta_m = m' - 1 - m_0\,,\quad \Delta_n = n' - \ell_0 - n_0\]
 for the numbers of paths and vertices sampled uniformly without replacement en route.
 Again, the maximal path length assumption implies that
 \begin{equation}
   \label{eq-n0-n'-y}
    n'-y- 5\log_{2-\delta}k < n_0 \leq n'-y\,,
 \end{equation}
 whence $\Delta_n  = y-\ell_0 + O(\log k)$.

We now claim that $|\Delta_m - \frac{m'-1}{n'-\ell_0} (y-\ell_0)| \leq k^{3/4}$ except with probability, say, $O(k^{-100})$.
Indeed, the sum $\Sigma$ of the lengths of the paths $2,\ldots,w$ sampled as above is hypergeometric with mean $(w-1)(n'-\ell_0)/(m'-1)$. Thus for a specific value of $w$ where $2\leq w\leq k$ with $|w - \frac{m'-1}{n'-\ell_0}(y-\ell_0)|>k^{3/4}$, and a
specific $v$ with $0 \leq v\leq 5\log_{2-\delta}k$, we have that
\[ \P(\Sigma = n'-y - v ) \leq \exp\bigg(-(1-o(1)) \frac{k^{3/2}}{2w}\bigg) \leq \exp\left(-(1/2-o(1))\sqrt{k}\right)\]
by Hoeffding's inequality (using our bound on the maximal path length). Summing this probability over $w$ and $v$ is easily $O(k^{-100})$, and we may assume henceforth that this event does not occur.

Finally, we wish to infer from Claim~\ref{clm:normal-abnormal} that the path distribution on the remaining $m_0$ paths (uniformly chosen out of the $m'-1$ paths that are left after positioning the leading size-biased one) is $\delta'$-normal for any $\delta'>\delta_0$ except with probability $O(k^{-100})$. This is achieved as follows:
\begin{compactitem}[\indent$\bullet$]
  \item If $r \geq \log k$ then $m_0\sim m'$ and so Part~\eqref{it-abnorm-m'-m} of that claim implies that the probability of a $\delta'$-abnormal path distribution is at most  $\exp(-k^{1-o(1)})$
      (even after enumerating over the possible values of $m_0$).
  \item If $r \leq \log k$ then $\log_{2-\delta} k\sim\log_{2-\delta} n'$, and since the maximal length of a path in $S$ is $o(\log_{2-\delta} n')$, we may appeal to Part~\eqref{it-abnorm-log-k} of that claim. This results in a probability of $\exp(-k^{\theta_0+o(1)})$ for a $\delta'$-abnormal path distribution, even after enumerating over the (polynomial in $k$) number of possibilities for $m_0$ and the total number of paths in our $m_0$-subset.
\end{compactitem}

We now proceed to choose paths uniformly from the remaining $m_0$ paths without replacement. Our target to hit is $k-\Delta_n$. Defining   $g$, $R$ and $f$ as for our application of Theorem~\ref{thm-renewals}(b) above,  we may appeal to the renewal without replacement estimate from Theorem~\ref{thm-renewals}(a), and deduce that the probability of hitting $k-\Delta_n$ is
\[\frac{m_0}{n_0} + O\left((2-\delta')^{-\left(k-(n'-n_0)\right)}\right) + O(1/n')\,.\]
It is easy to see (using~\eqref{eq-n0-n'-y}) that
\[\frac{m_0}{n_0} =\frac{m'-1-\Delta_m}{n'-\ell_0-\Delta_n}
 = \frac{m'-1}{n'-\ell_0} + O\left(\frac{k^{3/4}}{n'}\right) + O\left(\frac{\log k}{n'}\right)
%= \frac{m'-1}{n'-\ell_0} + O\left(\frac{k^{3/4}}{n'}\right)
\,.\]
Additionally,
\[ \frac{m'-1}{n'-\ell_0} = \frac{m'}{n'} + O\left(\frac{\ell_0}{n'}\right) = \frac{m'}{n'} + O\left(\frac{\log k}{n'}\right)\,,\]
and combining these gives that the hitting probability is $\frac{m'}{n'} + O(k^{3/4}/n')$.
That is, $\psi(r) = O(k^{3/4})$ in this case, as required.
\end{proof}

\subsection{Proof of Theorem~\ref{thm-N(S)} modulo Theorem~\ref{thm:normal}}
\label{subsec:N(S)-via-normal}

Fix $0 < \delta' < \delta < 1$.
For the reduction to $\delta$-normality (and applicability of Theorem~\ref{thm:normal}),
it will turn out that long paths, and
in particular violations of~\eqref{eq-normal-long}, are our main concern,
since the short-path condition~\eqref{eq-normal-short}, even in its more restrictive
version with $\delta'$, is satisfied with high enough probability that we may more
or less ignore path distributions that violate it.
We will deal with the long paths by disposing of them (that is, choosing
the cycles containing them) first.
Of course, this entails making sure that this ``preprocessing"
(i) is affordable and (ii) doesn't (usually) too badly distort the
path distribution of what remains (a point that will exploit
the slack between $\delta'$ and $\delta$).

Recall that the short-path condition~\eqref{eq-normal-short} for $\delta'$ says that
\[ \left| p_\ell - 2^{1-\ell}\right| < \epsilon_\ell = \epsilon_\ell(\delta') = \frac{1}{\ell^4 (2-\delta')^\ell \log^{1/8} k} \]
for $\ell\leq\sM = \frac18\log\log k$. Notice that if $\left|x - 2^{1-\ell}\right| > \epsilon_\ell$ for some $\ell \leq \sM$, then we violate the estimate
$ \frac{2-\epsilon}{(2+\epsilon)^\ell} < x < \frac{2+ \epsilon}{(2-\epsilon)^\ell}$ (from~\eqref{eq-m-dev})
for $\epsilon = (\log n)^{-1/3}$, since the terms sandwiching $x$ also sandwich $2^{1-\ell}$ and differ by a factor $1+O(\epsilon \ell) = 1+(\log n)^{-1/3+o(1)}$, whereas
$\epsilon_\ell > (\log k)^{-1/4} > (\log n)^{-1/4}$ for large enough $n$.

Plugging the hypothesis $|\frac{m}n - \frac13| < (\log n)^{-1/3}$ of Theorem~\ref{thm-N(S)} into~\eqref{eq-m-dev}
establishes that the probability of violating the short-path condition
w.r.t.\ $\delta'$ is at most $\exp(-n k^{-o(1)} \log^{-2/3} n) \leq \exp(-n k^{-2/3-o(1)})$.
Comparing this to the factor $k^{n/k}$ in~\eqref{eq-N(S)-overestimate}, we see that the
total contribution to $\sum \sN(S)^2$ from intersection patterns violating the short-path condition is $o\left((m!2^m)^2|\inter_{0,m}|\right)$.
(We could alternatively appeal to Lemma~\ref{l:better} for a factor of $(\log k)^{3n/k}$ but that would make little difference here.)
In what follows, we may thus confine our attention to intersection patterns satisfying the short-path condition w.r.t.\ $\delta'$.

 Next, consider an intersection pattern $S$ violating the long-path condition w.r.t.\ $\delta'$. (Recall this says that $ p_\ell \leq \gamma_\ell(\delta') = 1/\left(\ell^{4}(2-\delta')^{\ell}\right)$ for every $\ell\geq \sM$.)

We first claim that, following the same line of argument used above, we can reduce to the situation where no $p_\ell$ violates this condition for any $\sM \leq \ell \leq \sM^*_{\delta'}$, with
\[ \sM^*_{\delta'} := \log_{2-\delta'}\left(k/\log^5 k\right)\,.\]
Indeed, for such $\ell$ we have $\gamma_\ell(\delta') \geq c\frac{\log k}k$ for $c=c(\delta')>0$, thus Eq.~\eqref{eq:PathsBound} with $t=m\gamma_\ell$ and (noting that $\ell \gamma_\ell \to 0$ as $\ell\to\infty$ and in particular $t\ell = o(m)$) yields
\[ \P\left(p_\ell \geq \gamma_{\ell}(\delta') \right) \leq \left( c' \ell^{4} (1-\delta'/2+o(1))^{\ell}\right)^{ c m\frac{ \log k}k} \leq e^{-c'' (n/k)\log k \log\log k}\,.
\]
As before, the probability of encountering such an $S$ nullifies its contribution to $\sN(S)^2$, even via the rough overestimate of $k^{n/k}$ from~\eqref{eq-N(S)-overestimate}.

At this point, only $\ell$-vertex paths with $\ell \geq \sM^*_{\delta'}$ can potentially violate the conditions required for $\delta$-normality.
 For such $\ell$ with $p_\ell\geq\gamma_\ell(\delta')$, label an arbitrary set of $\lceil (p_\ell - \gamma_\ell(\delta'))m \rceil$ paths of length $\ell$ as \emph{excess}.

 Roughly speaking, our goal at this point is to somehow eliminate the excess paths and reduce to a $\delta$-normal (as opposed to a $\delta'$-normal) path distribution, to which we can, finally, apply Theorem~\ref{thm:normal}.
A key point is that we can achieve this by revealing the cycles (of the $\rod$ factor) containing the excess paths \emph{first}.
The number of ways to choose these early cycles will be somewhat larger than we would like, but this is more than compensated for by the fact that excesses are rare.

Let $R=R(S,\delta')$ denote the number of excess paths in $S$.

\begin{claim}\label{clm:excess-dist}
Let $S\in\inter_{0,m}$ be a uniformly chosen intersection pattern and let $\delta'>0$. Then
  \[
  \P(R \geq r) \leq k^{-\frac{\delta'-o(1)}{2\log(2-\delta')} r }\quad\mbox{for any $r\geq 1$}\,.
  \]
\end{claim}

%\begin{claim}\label{clm:excess-dist}
%Let $S\in\inter_{0,m}$ be a uniformly chosen intersection pattern, let $\delta'>0$ and let $R=R(S,\delta')$ denote the number of excess paths in $S$. Then
%  \[
%  \P(R \geq r) \leq k^{-\frac{\delta'-o(1)}{2\log(2-\delta')} r }\quad\mbox{for any $r\geq 1$}\,.
%  \]
%\end{claim}
\begin{proof}
If there are $r$ excess paths then for some $\ell \geq \sM^*_{\delta'}$ there are at least $t = \lfloor m \gamma_\ell \rfloor + r$ paths of length at least $\ell$ in $S$; namely, for the least $\ell$ for which there is an excess, we have $\lfloor m \gamma_\ell \rfloor + s$ paths of length $\ell$
for some $s \geq 1$, and, if $s < r$, at least $r-s$ excess paths of length greater than $\ell$.
By~\eqref{eq:PathsBound}, the probability of this event is at most
\[ \left( \frac{4e m}{m \gamma_\ell \left(2-o(1)\right)^{\ell}} \right)^{r} = \left( 1-\delta'/2+o(1)\right)^{\ell r} \leq e^{(-\delta'/2 + o(1))\ell r}\,,
\]
where the first equality absorbed the factor $c \ell^4$ into the $o(1)$-term in the base of the exponent, justified by the fact that $\ell\to \infty$.
The fact $\ell\geq \sM^*_{\delta'} = (1+o(1))\log_{2-\delta'} k$ completes the proof.
\end{proof}
Claim~\ref{clm:excess-dist} allows us to reduce to the case
\begin{equation}\label{eq-R-n/k}
 R \leq \frac{n}{k\log\log k}\,,
\end{equation}
since applying it with $r=n/(k\log\log k)$ gives $\P(R\geq r) \leq e^{-c \frac{n}{k}\frac{\log k}{\log\log k}}$, which
outweighs the factor $(\log k)^{3n/k}$ from the bound on $\sN(S)$ given by Lemma~\ref{l:better}.
%
%In fact, more is true: by~\eqref{eq:PathsBound}, the probability that there are at least $t=n/(k\log\log k)$ paths (not just excess paths) of length at least $j=\sM^*(\delta')$ is at most
%\[ \left( \frac{c k\log\log k}{\left(k/\log^5 k\right)^{\log_{2-\delta'}2-o(1)}} \right)^{\frac{n}{k\log\log k}} \leq  e^{-c' \frac{n}{k}\frac{\log k}{\log\log k}}
%\]
%(where the last inequality holds for large $n$), again
%outweighing the factor $(\log k)^{3n/k}$ from the bound on $\sN(S)$ in Lemma~\ref{l:better}. In particular, the same holds for $j=\log_{2-\delta}k > \sM^*(\delta')$ (for large enough $k$, as $\delta>\delta'$), a fact which will later enable us to apply Theorem~\ref{thm:normal} to any $\delta$-normal subset of $S$.

\begin{claim}\label{clm:S-bar-abnormal}
Let $\delta>\delta'>0$, and suppose $S\in\inter_{0,m}$ satisfies the $\delta'$-normality conditions for path lengths $\ell\leq\sM^*_{\delta'}$ and
has $r \leq n/k$ excess paths.
For some $\bar{m} \sim m$ with $\bar{m}\le m-r$, let $\bar{S}$ be a uniform $\bar{m}$-subset of the $m-r$ non-excess paths in $S$. Then
\[ \P\left(\bar{S}\mbox{ is $\delta$-normal}\right) \geq 1-\exp\left(-\bar{m}k^{-o(1)}\right)\,.\]
\end{claim}
\begin{proof}
Let $U$ be the set of non-excess paths, let $u=m-r$ denote its size and let $\mu_\ell$ denote the relative frequency of $\ell$-vertex paths in $U$. We will show that, deterministically, $U$ is $\delta'$-normal for any $\delta'>\delta$, at which point the statement of the claim will follow immediately from Part~\eqref{it-abnorm-m'-m} Claim~\ref{clm:normal-abnormal} since $\bar{m}\sim u$ (both are asymptotically $m$).

First consider the long-path condition. By our definition of excess paths, for any $\ell\geq \sM$ we have at most $m \gamma_\ell(\delta) \sim u \gamma_\ell(\delta)$ of length $\ell$. As in the proof of Claim~\ref{clm:normal-abnormal}, we have $ \gamma_\ell(\delta)/\gamma_\ell(\delta') = %\left(\frac{2-\delta'}{2-\delta}\right)^{\ell} =
\big(1 + \frac{\delta - \delta'}{2-\delta}\big)^{\ell}$, which diverges with $\ell$ (as a poly-log even), easily implying that $\mu_\ell < \gamma_\ell(\delta')$ for large enough $n$.
To consider short paths, take $\ell \leq \sM$. By hypothesis, $\left| p_\ell - 2^{1-\ell} \right| < \epsilon_\ell(\delta')$ in $S$.
Since by definition all excess paths have length greater than $\sM^*_{\delta'}$, the set $U$ contains all the $\ell$-vertex paths in $S$ and exactly $m(1-p_\ell)-r$ others. It follows that $\mu_\ell$, the  relative  frequency of $\ell$-vertex paths in $U$, satisfies
\begin{equation*}
%   \label{eq-mu-p}
   \left|\mu_\ell - p_\ell\right| \leq r/m = O(1/k) = o\left(\epsilon_\ell(\delta')\right)\,,
 \end{equation*}
where the last inequality holds for $\ell\leq\sM$ since $\epsilon_\ell(\delta')\geq\epsilon_{\sM}(\delta') = k^{-o(1)}$.
The fact that $ \epsilon_\ell(\delta)/ \epsilon_\ell(\delta') = \big(1+\frac{\delta-\delta'}{2-\delta}\big)^\ell > 1$ now implies the short-path condition and completes the proof.
\end{proof}

\begin{claim}\label{clm:excess-rod}
Let $\delta'>0$, and suppose $S\in\inter_{0,m}$ satisfies the $\delta'$-normality conditions for path lengths $\ell\leq\sM^*_{\delta'}$ and has $r \leq n/(k\log\log k)$ excess paths.
Then
  \[
  \sN(S) \leq (3+o(1))^{r+1} 2^{m} m!\,.
  \]
\end{claim}
\begin{proof}
To bound the number of $\rod$ factors that $S$ gives rise to, we first account for the cycles containing the $r$ excess paths, and then bound the number of ways to fill in the remaining cycles and produce a $\rod$ factor.

We will use $q$ to denote the number of cycles that contain excess paths in the $\rod$ factor, and let $x$ denote the total number of paths in these $q$ cycles. Note that trivially
\[ \lceil r/k \rceil \leq q \leq r\qquad\mbox{ and }\qquad r \leq x \leq q k = o(m)\,.\]

We will begin by counting all the $\rod$ factors in which, once we filter out the above mentioned $q$ cycles containing all excess paths, the remaining $m-x$ paths form a $\delta$-abnormal distribution. For given $q$ and $x$, one first chooses these $m-x$ paths out of the $m-r$ possible non-excess ones. As $x = o(m)$, Claim~\ref{clm:S-bar-abnormal} shows that an $e^{-m k^{-o(1)}}$ fraction of these choices results in a $\delta$-abnormal path distribution. For each such choice, there are at most $2^{m-x} (m-x)! k^{n/k-q} $ ways to order and orient the paths, and then root the cycles. As for the remaining $q$ cycles, we select an excess path for each of these, order and orient the remaining $x-q$ paths, then root the cycles and insert them into the list of $n/k$ cycles with a final factor of $(n/k)_q$. The number of such $\rod$ factors (i.e., with a $\delta$-abnormal suffix) is thus at most
\begin{align}
 & e^{-m k^{-o(1)}}\binom{m-r}{x-r} \binom{r}{q} (x-q)! \left(\frac{n}k\right)_q (m-x)! k^{n/k} 2^m \label{eq-excess-abnormal-suffix}\\
 = &e^{-m k^{-o(1)}} \binom{r}q \frac{(m-r)_{x-r}}{(m)_x} \frac{(x-q)!}{(x-r)!}  \left(\frac{n}k\right)_q k^{n/k} m! 2^m \,.\nonumber%\label{eq-abnormal-suffix-1}
\end{align}
As $(m)_x/(m-r)_{x-r} = \left((\frac13-o(1))n\right)^r$, this expression (slightly rearranged) is at most
\[
  e^{-m k^{-o(1)}} 2^r \left(\frac{3+o(1)}k\right)^q \left(\frac{(3+o(1))x}n\right)^{r-q} k^{n/k} m! 2^m \,.\nonumber%\label{eq-abnormal-suffix-1}
 \]
As $2^r k^{n/k} = e^{O((n/k)\log k)}$, even after summing over $q$ and $x$ this is $o( m!2^m)$. We may thus restrict our attention to $\rod$ factors for which the cycles that do not contain excess paths induce a $\delta$-normal distribution.

To count these, we number our excess paths from 1 to $r$ (in an arbitrary way) and proceed as follows:
\begin{enumerate}[(i)]
  \item \label{it-alg-part-1} Repeat the following steps until all excess paths are exhausted:
 \begin{itemize}
  \item Select a location for a new cycle (amongst the $n/k$ slots), the lowest numbered remaining excess path to be a part of it (which we direct, as usual), and its root (given by its offset  $i\in[k]$ from the start of the chosen path).
  \item Complete the cycle that contains this path (including directions) using any choice of the remaining paths, including excess ones.
\end{itemize}
  \item \label{it-alg-part-2} Order (and direct) the remaining paths.
\end{enumerate}
This gives an arrangement consisting of \begin{inparaenum}[(i)]
  \item a set of rooted, directed cycles containing all excess paths (where each cycle contains at least one excess path), each with its position in one of $n/k$ cycle slots specified,
       and
  \item an ordering of the remaining (non-excess) paths.
\end{inparaenum}
We claim the total number of arrangements is at most
\begin{equation}
  \label{eq-algorithm-term}
(3+o(1))^r m! 2^{m}  \,.
\end{equation}
Indeed, there are $m-i+1$ choices for the $i$-th path unless we are at the beginning of a new cycle, say the $j$-th cycle. In the latter case, the next path is dictated by the ordering of the excess paths, while we have $(n/k-j)k$ possibilities for positioning and rooting the new cycle. Since the excess paths are exhausted after some $q \leq r$ cycles, at which point the number of remaining paths is at least $m - q k \geq m - r k = (1-o(1))m$, we can replace each such term $(n/k-j)k$ by the ``default'' term $m-i+1$ at a cost of $(1+o(1))m/n = 3+o(1)$. This establishes~\eqref{eq-algorithm-term}.

To complete the proof, partition the arrangements into classes according to the output of Step~\eqref{it-alg-part-1}. For any given class, if $x$ is the total number of paths in its set of excess cycles, by definition there are $(m-x)!2^{m-x}$ arrangements in the class, corresponding to Step~\eqref{it-alg-part-2}. On the other hand, we may assume the remaining $m-x$ paths have a $\delta$-normal distribution by the previous discussion. Thus, we know by Theorem~\ref{thm:normal} that the number of $\rod$ factors whose cycles with excess paths agree with this class, is at most $ (3+o(1))(m-x)!2^{m-x}$, so at most $3+o(1)$ times the number of arrangements in the class. Altogether, $\sN(S)$ is at most $3+o(1)$ times the total number of arrangements, as required.
\end{proof}

We are now in a position to complete the proof of Theorem~\ref{thm-N(S)}.
 Let $\cE=\cE(S)$ be the event that $S\in\inter_{0,m}$
satisfies the $\delta'$-normality conditions for path lengths $\ell\leq\sM^*_{\delta'}$ and has at most $ n/(k\log\log k)$ excess paths. Thus far, we have reduced the proof of the theorem to showing that
\[ \E\left[ \sN(S)^2 \one_{\cE}(S)\right] \leq (3+o(1))m!2^m\,,\]
  where $S$ is a uniformly chosen intersection pattern. Combining Claim~\ref{clm:excess-rod} with the estimate in Claim~\ref{clm:excess-dist} on the number of excess paths $R$, we get
\begin{align*}
  \E&\left[ \sN(S)^2\one_{\cE}(S)\right] = \sum_r \E\left[ \sN(S)^2\one_{\cE}(S)\mid R=r\right]\P(R=r)   \\
  &\leq (3+o(1))m!2^m \bigg(1 + \sum_{r\geq 1} \left[(3+o(1))k^{-c(\delta')}\right]^r \bigg) = (3+o(1))m!2^m
\end{align*}
(where $c(\delta')$ is the constant from Claim~\ref{clm:excess-dist} and the summand for $r=0$ is given by Theorem~\ref{thm:normal} as $S$ is then $\delta'$-normal), as required.
\qed

% Notes we do not seem to need.
%(Note that we can assume there are no paths of length greater than $k$ as otherwise $\sN(S)=0$.)
%(Note that no non-excess path has length greater than $\log_{2-\delta'} n$.)

\subsection{Abnormal suffix: proof of Theorem~\ref{thm:normal}}\label{subsec:bad-suffix}
For an inductive proof based on placing cycles, the main problem is to bound number of ways a $\rod$ factor can be formed such that the $t$-suffix  is $\delta$-normal for each $t>i$ and yet $\delta$-abnormal for $t=i$.  These two events (normal when more than $i$ cycles remain, and abnormal when $i$ remain) each have very small probability and we cannot assume independence. Moreover, given an abnormal $i$-suffix, a complementary prefix (of total length $n-ik$) might not be $\delta'$-normal  as a stand-alone path distribution for any useful $\delta'$. These features make this argument rather twisted.

By the $\delta$-normality of $S$, the number of paths that are of length at least $\log_{2-\delta} k$
is at most $m\sum_{\ell \geq \log_{2-\delta} k}\gamma_\ell(\delta)$, which is $O(m / (k \log^4 k)) < n/(k\log k)$ for large enough $n$.
 The first step in the proof is to reduce to the case where there are no such paths whatsoever, after which the following lemma would complete the proof.
\begin{lemma}\label{lem:normal-no-long}
For any fixed $\delta>0$, if $S\in\inter_{0,m}$ is $\delta$-normal and contains no path of length at least $\log_{2-\delta}k$ then
\[ \sN(S) \leq (3+o(1))m!2^m\,.
\]
\end{lemma}
We postpone the proof of the above lemma in favor of first showing how to reduce to its setting, in a way similar to our treatment of excess paths in the proof of Claim~\ref{clm:excess-rod}.

Fix $\delta'>\delta$. Let $T$ be the subset of all paths of length at least $\log_{2-\delta}k$ in $S$, and let $t=|T| \leq n/(k\log k)$. We first consider all the $\rod$ factors containing $S$ for which there is a subset $C_1,\ldots,C_{q-1}$ of the cycles and $C_q$, which is part (possibly all) of another cycle, such that every $C_i$ contains some path in $T$ and the distribution of $S \setminus \cup_{i\leq q} C_i$ is $\delta'$-abnormal.
To estimate the number of such factors, we sum over $q \leq t = o(n/k)$ and select $q$ such paths (each for a separate cycle). Next, we sum over $x \leq k q = o(m)$, the total number of paths in these cycles, and run over all $\binom{m-q}{m-x}$ subsets for the remaining cycles.
The combination of Claim~\ref{clm:normal-large-subset} and Part~\eqref{it-abnorm-m'-m} of Claim~\ref{clm:normal-abnormal}
implies that
for any $\delta'>\delta$ an $O(\exp(-m k^{-o(1)}))$-fraction of these will result in a $\delta'$-abnormal path distribution
(the former addresses the normality of the $(m-q)$-subset and the latter treats its $(m-x)$-subsets).
Thus, similarly to Eq.~\eqref{eq-excess-abnormal-suffix}, the total number of such $\rod$ factors is at most
\begin{align}
 & e^{-m k^{-o(1)}}\binom{t}q (m-q)_{x-q} \left(\frac{n}k\right)_q (m-x)! k^{n/k} 2^m \nonumber\\
  \leq &e^{-m k^{-o(1)}} 2^t \bigg(\frac{3+o(1)}k\bigg)^q k^{n/k} m! 2^m \,,\nonumber
\end{align}
where we used the fact that $m-i = (1-o(1))m$ for $1\leq i \leq x$ to replace the term $n^q$ by $(3+o(1))^q$. Since $q\leq t$ and $(6+o(1))^t k^{n/k} = \exp(m k^{-1+o(1)})$, the entire expression is $o(m! 2^m)$, as required.

The other $\rod$ factors (in which no set $C_1,\ldots,C_q$ as above leaves behind a $\delta'$-abnormal path distribution) are handled by prioritizing the paths of length at least $\log_{2-\delta}k$, as done before with the excess paths for the proof of Claim~\ref{clm:excess-rod}. This time, however, it is crucial to estimate the probability that these align to $k$-cycles, because we cannot afford to give away any constant factor (let alone a larger term such as the $(3+o(1))^q$ in~\eqref{eq-algorithm-term}).
Using the same procedure as in that claim (order the paths in $T$ arbitrarily, repeatedly select the lowest numbered such path and complete it into a cycle, and finally order the remaining paths), we now argue that in lieu of the estimate~\eqref{eq-algorithm-term}, the total number of arrangements is at most
\[ (1+o(1))m!2^m\,.\]
To see this, as in the proof of Theorem~\ref{thm:goodSuffixes}, we appeal to one of two strategies depending on the relation between $k$ and $n$:
\begin{compactitem}[\indent$\bullet$]
  \item If $k = o(\sqrt{n})$, we will appeal to our renewal estimate after placing each of the $q$ leading paths in $C_1,\ldots,C_q$. Namely, recall that upon forming the $(j+1)$-st cycle with a leading path from $T$ --- letting $\ell_0$ denote the length of this path --- the set of remaining paths (excluding the leading path) is $\delta'$-normal by our current assumption. Implementing Theorem~\ref{thm-renewals}(a) as before, we see that the probability of hitting the partial sum $k-\ell_0$ in a random permutation over the remaining elements is $\frac{m'-1}{n'-\ell_0} + O(1/n') + O((2-\delta')^{-(k-\ell_0)})$, where $m'$ and $n'=n-j k$ are the numbers of paths and vertices left after the first $j$ cycles, respectively. Thus, the $n'=(n/k-j)k$ choices for positioning and rooting the cycle can be replaced by the ``default'' term $m'$ at a multiplicative cost of $1+O(\ell_0/n') = 1+O(k/n)$ (as the total number of vertices in these $q$ cycles is at most $q k \leq t k = o(n)$). Repeating this procedure for all $q\leq t$ cycles, then finally ordering and directing all remaining paths, bounds the number of arrangements by $\exp(O(t k/ n))m!2^m = (1+o(1))m!2^m$, as claimed.
  \item If $k \gtrsim \sqrt{n}$, we tweak the above approach by revealing the paths that follow the leading path in $C_{j+1}$ until reaching at least $y=\lfloor k-10\log_2 k\rfloor$ vertices in that cycle. Let $m_0$ and $n_0$ be the numbers of paths and vertices remaining at that point. By our assumption, these remaining paths form a $\delta'$-normal distribution; thus, the renewal estimate from Theorem~\ref{thm-renewals}(a) is $m_0/n_0 + O(1/n')$. As $n_0 = n' - y + O(\log k)$, this is $m'/n' +O(k^{3/4}/n') + O(k^{-100})$ due to the variability in $m_0$. Hence, we are again entitled to replace the term $n'=(n/k-j)k$ by $m'$, this time at a cost of $1+O(\ell_0/n')+O(k^{3/4}/n')$. Accumulating these errors over the $q\leq t$ cycles gives a factor of $\exp(O(tk / n) + k^{-1/4+o(1)})=1+o(1)$, and therefore a total of at most $(1+o(1))m!2^m$ arrangements.
\end{compactitem}
The final step is to partition the arrangements into classes according to the cycles involving $T$; if these contain a total of $x$ paths in some given class then this class contains $(m-x)!2^{m-x}$ arrangements. The path distribution on the remaining $m-x$ paths is $\delta'$-normal and contains no path of length at least $\log_{2-\delta} k$, and so Lemma~\ref{lem:normal-no-long} guarantees that the number of $\rod$ factors agreeing with this class is at most $(3+o(1))(m-x)!2^{m-x}$. Altogether, $\sN(S) \leq (3+o(1))m!2^m$ modulo Lemma~\ref{lem:normal-no-long}.

With the values of $n$ and $k$ understood, we say that a path distribution $P$ on $n-kr$ points is {\em $r$-close to $\delta$-normal} if there is some set of paths on a total of $rk$ points which, when added to $P$, gives a $\delta$-normal path distribution   with maximum length at most $\log_{2-\delta} k$ (in particular, no paths of greater length exist in $P$ itself). Such path distributions are the ones that can conceivably result when deleting (the last) $r$ cycles from the path distributions under consideration.

To complete the proof of Lemma~\ref{lem:normal-no-long} we will use the following.

\begin{lemma}\label{l:rClose}
Fix $\delta'>\delta>0$ and let $0<\theta<\theta_0$ for $\theta_0=\theta_0(\delta,\delta')$ as was given in Claim~\ref{clm:normal-abnormal}. For all $r=r(n)\geq 1$,
\begin{enumerate}[(i)]
  \item \label{it-r-close}
  If $S'$ is a set of $m'$ paths on $n-rk$ vertices that is $r$-close to $\delta$-normal, then for any sufficiently large $n$, the number of $\rod$ factors (consisting of $n/k-r$ cycles) that contain $S'$ is at most
$(m')!2^{m'}\exp\left( r k^{\theta}\right)$.
\item \label{it-bad-suffix} For any $\delta$-normal $S\in\inter_{0,m}$ with no paths of length at least $\log_{2-\delta}k$, if $n$ is large enough then the number of $\rod$ factors arising from $S$ which  have a   $\delta'$-abnormal   $r$-suffix  is at most
$ m!2^m\exp\left(- r k^{\theta}\right)$.
\end{enumerate}
\end{lemma}
(For our purposes, the bound in Part~\eqref{it-bad-suffix} could be replaced by any other that would give $o(2^m m!)$ when summed over $r$.)

\begin{proof}
We first prove Part~\eqref{it-r-close} by downward induction on $r$. Let
\[ r_0 = \frac{n}{k^{1+\theta}}\log k\,. \]
First consider $r\ge r_0$.  There are $(m')!2^{m'}$ possibilities for ordering and directing the paths (ignoring the need to hit multiples of $k$). The choices of roots in the lead paths give an extra factor at most $(\log_{2-\delta} k )^{n/k}$ by the upper bound on path length. Since $r\ge r_0$, the claimed upper bound is at least $(m')!2^{m'} \exp [(n/k)\log k]$ and so Part~\eqref{it-r-close} holds in this case.

Now take $1\leq r < r_0$ while assuming that Part~\eqref{it-r-close} holds for all $r < r' \leq r_0$.
Fix $\tilde{\delta} > \delta'$ and, for $\tilde{r}\geq 1$, let $\cR_{\tilde r}=\cR_{\tilde r}(S')$ be the set of $\rod$ factors arising from $S'$ where the $\tilde{r}$-suffix is $\tilde{\delta}$-abnormal but all shorter suffixes are $\tilde\delta$-normal. (It suffices to treat this case since the number of $\rod$ factors in which all suffixes are $\tilde\delta$-normal is $\sN_{\tilde\delta}(S) \leq (3+o(1))(m')!2^{m'}$ by Theorem~\ref{thm:goodSuffixes}, a fraction of $O(\exp(-k^\theta))=o(1)$ out of the desired upper bound.)

As $S'$ is $r$-close to $\delta$-normal, it is deterministically $\delta'$-normal by Claim~\ref{clm:normal-large-subset} using $r k \leq r_0 k = n k^{-\theta+o(1)}$.
Claim~\ref{clm:normal-abnormal} therefore implies (via Part~\eqref{it-abnorm-log-k} of that claim, noting that $S$ has no paths of length at least $\log_{2-\delta'} k \geq \log_{2-\delta}k$) that if $\tilde S$ is a uniform $\tilde m$-subset of $S'$ and $\tilde\Sigma$ is its number of vertices then
\begin{equation}
  \label{eq-tilde-S-abnormal}
  \P\left(\mbox{$\tilde S$ is $\tilde \delta$-abnormal},\,\tilde \Sigma=\tilde{r}k\right) \leq \exp\left(- \tilde{r} k^{\theta_0-o(1)}\right)
\end{equation}
(where we absorbed the prefactor $\sqrt{\tilde r k}$ from that bound into the $o(1)$-term).
We will now argue that the number of $\rod$ factors in $\cR_{\tilde r}$ that have $\tilde{m}$ paths in the $\tilde{r}$-suffix and $x$ paths in the $(n/k-\tilde{r}+1)$-st cycle is at most
\begin{align}
&\bigg[\binom{m'}{\tilde m} e^{-\tilde r k^{\theta_0-o(1)}}\bigg] \bigg[(m'-\tilde m)! 2^{m'-\tilde{m}}e^{(r+\tilde r)k^\theta}\bigg]\nonumber \\
&\cdot\bigg[(\tilde m)_x 2^x \log_{2-\delta}k\bigg]\bigg[(3+o(1))(\tilde m-x)!2^{\tilde m - x}\bigg]
\,.\label{eq-R-tilde-r}
 \end{align}
The first expression in brackets corresponds to choosing $\tilde{m}$ paths for the $\tilde\delta$-abnormal $\tilde r$-suffix, as estimated above. The second expression bounds the number of ways to form a $\rod$ factor out of the remaining $m'-\tilde m$ paths (the first $n/k-\tilde r$ cycles) via the induction hypothesis (since the set of $m'-\tilde{m}$ paths in the the $n/k-\tilde r$ prefix are $(r+\tilde r)$-close to $\delta$-normal and $\tilde r \geq 1$). The third expression treats the $(n/k-\tilde r+1)$-st cycle, namely, ordering and directing its paths and selecting a root out of the first path (whose length is at most $\log_{2-\delta}k$). Finally, the last expression treats the $(\tilde r-1)$-suffix, in which all suffixes are $\tilde\delta$-normal by assumption, via an application of Theorem~\ref{thm:goodSuffixes}.

  Since $\tilde{m} \leq \tilde{r}k$ and $x\leq k$, rearranging~\eqref{eq-R-tilde-r} gives, for each $\tilde{r}\geq 1$,
\begin{align*}
   |\cR_{\tilde{r}}| &\leq \tilde{r} k^2 (m')! 2^{m'} e^{r k^{\theta} -\tilde r (k^{\theta_0-o(1)} - k^{\theta})} (3+o(1))\log_{2-\delta}k \\
 &= e^{-\tilde r k^{\theta_0 -o(1)}}(m')! 2^{m'} e^{r k^\theta}
 \end{align*}
(using $\theta_0 > \theta$), and summing this over $\tilde{r}\geq 1$ now gives $o\big((m')! 2^{m'} e^{r k^\theta}\big)$.
This establishes Part~\eqref{it-r-close}.

It remains to prove Part~\eqref{it-bad-suffix}.
Since $S$ is $\delta$-normal, similar to~\eqref{eq-tilde-S-abnormal},
the number of $m'$-subsets of $S$ which are $\delta'$-abnormal and contain exactly $r k$ vertices is at most
\begin{equation}
  \label{eq-delta'-abnormal-m'}
  \binom{m}{m'} \exp\left(-r k^{\theta_0-o(1)}\right)\,.
\end{equation}

First, take $r \geq r_0$. By ordering the $m'$ paths of the suffix as well as the remaining $m-m'$ paths, directing all $m$ paths and choosing a root for each cycle at a multiplicative cost of
$(\log_{2-\delta}k)^{n/k}$ we find that the total number $\rod$ factors arising from $S$ and having a $\delta'$-abnormal $r$-suffix is at most
\[ r k \cdot 2^m m! (\log_{2-\delta}k)^{n/k} \exp\left(-r k^{\theta_0-o(1)}\right)\,,\]
where the prefactor $rk$ bounds the number of choices for $m'$.
Since $r \geq r_0$, we have $r k^{\theta_0-o(1)} \geq n k^{-1+\theta_0-o(1)}$, which outweighs the $\exp(n k^{-1+o(1)})$
factor from rooting the cycles (as well as the factor $rk$), so the above upper bound is at most
\[ 2^m m! \exp\left(-rk^{\theta_0-o(1)}\right) < 2^m m! \exp\left(-rk^{\theta}\right)\]
for large enough $n$, as required.

When $r<r_0$, for each of the choices for a $\delta'$-abnormal $r$-suffix with $m'$ paths (as estimated in~\eqref{eq-delta'-abnormal-m'})
we order and direct the paths of the suffix, then root its cycles at a multiplicative cost of $(m')!2^{m'}(\log_{2-\delta} k)^r$.
As for the first $n/k-r$ cycles, the $m-m'$ paths used for these induce a path distribution with is $r$-close to $\delta$-normal; thus,
Part~\eqref{it-r-close} bounds the number of $\rod$ factors arising from these by $(m-m')!2^{m-m'}\exp(r k^{\theta})$.
Overall we get the upper bound
\[ r k \cdot 2^m m! \exp\left(r(k^{\theta}-k^{\theta_0-o(1)})\right) (\log_{2-\delta}k)^{r} \leq 2^m m! e^{-r k^\theta}\]
for large $n$ (where the prefactor $rk$ again accounts for the choice of $m'$). This establishes Part~\eqref{it-bad-suffix} and completes the proof of the lemma.
\end{proof}

Finally, returning to the proof
of Theorem~\ref{thm:normal}, we sum the expression from Lemma~\ref{l:rClose}\eqref{it-bad-suffix} over $r\geq 1$ and find that the contribution to $\sN(S)$
of $\rod$ factors with $\delta'$-abnormal suffixes is $o(m!2^m)$. Consequently, Theorem~\ref{thm:goodSuffixes} implies the statement of
Lemma~\ref{lem:normal-no-long}, which, as already noted, completes the proof of Theorem~\ref{thm:normal}.
\qed

\section{Second moment of cycles factors via Theorem~\ref{thm-N(S)}}
\label{sec:h=0-reduction}
Our main result in this section is the promised upper bound on $\E[\cfact^2]$ (see~ \S\ref{sec-second-moment-framework}), which is based on our estimate for the number of $\rod$ factors arising from cycle-free intersection patterns.
\begin{theorem}
  \label{thm-second-moment}
If $k \geq K_0(n)$ with $K_0(n)$ as in~\eqref{eq-K0}, then
the number of $k$-cycle factors in $G\sim\cP(n,3)$ satisfies
$\E[\cfact^2] \leq (3+o(1))\E[\cfact]^2$.
\end{theorem}
\begin{proof}
Since $\cfact / Y_k = (n/k)!(2k)^{n/k}$ (see~\eqref{eq-rod-rescale}), we need to show that
\[ \E[Y_k^2] \leq (3+o(1))\E[Y_k]^2\,.\]
Recall also (see~\eqref{eq-E[Y^2]}) that estimating $\E[Y_k^2]$ amounts to estimating
\begin{equation}
  \label{eq-sum-h-m}
  \sum_h \sum_m \match(n-2m) \sum_{S \in \inter_{h,m}} \sN(S)^2 \,.
\end{equation}
We begin with what will turn out to be the dominant contributions to this sum, those with $h=0$ and $m$ close to $n/3$. (The analysis of this simpler case is similar to the corresponding analysis of Hamilton cycles in cubic graphs~\cite{RW4}*{Theorem 2.4} and in $r$-regular graphs~\cite{FJMRW}, where $h=0$ by definition.)
Set
\[ \mathcal{J} \deq \left\{ \lceil(\tfrac13 -\delta)n\rceil, \ldots, \lfloor(\tfrac13+\delta) n\rfloor\right\}\quad\mbox{ where }\delta = \tfrac12(\log n)^{-1/3}\,, \]
and recall that Theorem~\ref{thm-N(S)} says that if $|\frac{m}n-\frac13|<(\log n)^{-1/3}$, then
\begin{align}
\sum_{S \in \inter_{0,m}} \sN(S)^2 &\leq (9+o(1))\left(m!\,2^m \right)^2 |\inter_{0,m}|\label{eq-Im-sum-upper-bound}
\end{align}
(the slight difference between the bounds on $|\frac{m}n-\frac13|$, here and in $\cJ$, will be helpful later).

Setting
\begin{equation}
  \label{eq-def-Psi0}
  \Psi_0(m) = \match(n-2m) n!\, m!\,2^m \binom{n- m-1}{m-1}
\end{equation}
and recalling from~\eqref{eq:I0m} that $ |\inter_{0,m}| = \frac{n!}{m!\, 2^m} \binom{n-m-1}{m-1}$, we rewrite~\eqref{eq-Im-sum-upper-bound} as
\[  \match(n-2m) \sum_{S \in \inter_{0,m}} \sN(S)^2 \leq (9+o(1))  \Psi_0(m)\,. \]
Notice that for any $m$,
\begin{equation*}
 % \label{eq-ratio-Psi(m)-Psi(m+1)}
  \frac{\Psi_0(m)}{\Psi_0(m+1)} = \frac{m(n-m-1)}{2(m+1)(n-2m)}\,.
\end{equation*}
If $x = 1/3 - m/n > 0$ then this is at most $ 1 - \frac{9x}{2(1+6x)}$, and so in this case
for any $m = n/3 - \gamma n$ with $0<\gamma<1/3$,
\begin{align}
\frac{\Psi_0(m)}{\Psi_0(\lfloor \tfrac{n}3 \rfloor)} &\leq
e^{-\frac92 \left(\int_{0}^{\gamma}\frac{x dx}{1+6x}\right)n} = e^{-\left[\frac{3}4\gamma - \frac{1}{8}\log(1+6\gamma)\right]n}
\leq e^{-\left(\frac{9}4\gamma^2 - 9\gamma^3\right)n}
\,.\label{eq-Psi(m)-lower}
\end{align}
Similarly, if $x = 1/3-m/n<0$ then
\[ \frac{\Psi_0(m+1)}{\Psi_0(m)} \leq (2+O(1/n))\frac{n-2m}{n-m} =(1+O(1/n))\left(1-\frac{9|x|}{2-3|x|}\right)\,,\]
whence, for $m=n/3+\gamma n$ with $0<\gamma<1/6$ (recall $m\leq n/2$),
\begin{align}
\frac{\Psi_0(m)}{\Psi_0(\lfloor \tfrac{n}3 \rfloor)} &\lesssim
e^{-9 \left(\int_{0}^{\gamma}\frac{x dx}{2-3x}\right)n} = e^{\left[3 \gamma + 2\log(1-\frac32\gamma)\right]n  }
\leq e^{-\frac{9}4\gamma^2 n }
\label{eq-Psi(m)-upper}
\end{align}
(where the multiplicative constant implicit in the first inequality accounts for accumulating the $1+O(1/n)$ error factor over $O(n)$ values of $m$).
Finally, in both cases, when $x=o(1)$ the upper bounds are essentially tight (using $1-x \geq \exp[-x/(1-x)]$, for instance); namely,
\begin{align}
\frac{\Psi_0(m)}{\Psi_0(\lfloor \tfrac{n}3 \rfloor)} &= e^{-(\frac{9}4+o(1))\gamma^2 n }\qquad\mbox{ for $\gamma=o(1)$}\,.
\label{eq-Psi(m)-middle}
\end{align}
From this last estimate we have
\begin{align}
\sum_{m\in\mathcal{J}} \Psi_0(m) &\sim \bigg(\int_{-\infty}^{\infty} e^{-(\frac94-o(1)) x^2 / n} dx \bigg) \Psi_0(\lfloor\tfrac{n}3\rfloor) %\nonumber\\
\sim \frac23 \sqrt{\pi n}\Psi_0(\lfloor\tfrac{n}3\rfloor)\,.
\label{eq-middle-sum}
 \end{align}

We next want to show that the remaining terms in~\eqref{eq-E[Y^2]} (those with $m\notin \cJ$ or $h\neq 0$) are negligible in comparison
with those in~\eqref{eq-middle-sum}.

For $h\neq 0$, we may choose $S\in\inter_{h,m}$ by first choosing the $h$ cycles --- this can be done in $\frac{(n)_{kh}}{h! (2k)^h}$ ways ---
and then choosing the set, say $S'$, of paths of $S$ from the remaining $n'=n-kh$ vertices. We then have
\[ \sN(S) = (n/k)_h (2k)^h \sN(S')\,, \]
where the first two factors count the number of ways to order (in the $\rod$ factor), root and direct the cycles.
Therefore, with $\inter_{0,m}^{h}$ the number of cycle-free, $m$-path intersection patterns on $n-kh$ vertices, we have
\begin{equation}
  \label{eq-sum-inter-0m-h}
  \sum_{S\in\inter_{h,m}} \sN(S)^2 = \frac{(n)_{kh}}{h! (2k)^h} \left( (n/k)_h (2k)^h\right)^2 \sum_{S'\in\inter_{0,m}^h} \sN(S')^2\,.
\end{equation}
The natural benchmark for $\match(n-2m) \sum_{S'\in\inter_{0,m}^h}\sN(S')^2$, analogous to $\Psi_0(m)$ (see~\eqref{eq-def-Psi0}), is
\[ \match(n-2m)  (n-kh)! m! 2^m \binom{n-kh-m-1}{m-1}\,.\]
Combining the last two displays, we find that the counterpart of $\Psi_0(m)$ in consideration of the contribution of $\inter_{h,m}$ is
\begin{align}
  \Psi_h(m) &\deq \match(n-2m) n!\, m!\,2^m \binom{n-k h - m-1}{m-1}\binom{n/k}{h}^2 (2k)^h h!\label{eq-Psi-def}\,,
\end{align}
though here we will sometimes need to retreat to
\begin{align}
    \hat{\Psi}_h(m) &\deq \Psi_h(m) k^{2(n/k - h)
    }\,.\label{eq-Psibar-def}
\end{align}
(It may be worth observing, though we do not make explicit use of this, that the preceding description of the number of choices for the cycles
of $S\in\inter_{h,m}$, combined with~\eqref{eq:I0m}, gives $|\inter_{h,m}| = \frac{1}{h! (2k)^h}\frac{n!}{m!\, 2^m} \binom{n-kh-m-1}{m-1}$.)

As we will soon see, the main task remaining is to control $\sum_{h>0}\!\sum_{m}\! \Psi_{h}(m)$. This is achieved by the following lemma.
\begin{lemma}
  \label{lem-Psi-bounds}
For $\Psi_h(m),\hat{\Psi}_h(m)$ as in~\eqref{eq-Psi-def} and $h^* = \lceil n/(k\sqrt{\log k})\rceil$,
\begin{align}
  \sum_{h=1}^{h^*}\sum_{m\in \cJ} \Psi_h(m) &= o\bigg(\sum_{m\in\cJ}\Psi_0(m)\bigg)\,,\label{eq-Psi-h-nonzero-sm}\\
    \sum_{h = h^*}^{n/k}\sum_{m\in \cJ} \hat{\Psi}_h(m)  &= o\bigg(\sum_{m\in\cJ}\Psi_0(m)\bigg)\,,\label{eq-Psi-h-nonzero-lg}\\
  \sum_{h=0}^{n/k}\sum_{m\notin \cJ} \hat{\Psi}_h(m)  &= o\bigg(\sum_{m\in\cJ}\Psi_0(m)\bigg) \,.\label{eq-Psi-m-notin-J}
\end{align}
\end{lemma}
The reason for the division at $h^*$ is that for $h< h^*$ we will be able to apply Theorem~\ref{thm-N(S)} (with $n-kh$ vertices)
to say that the inner sum on the left of~\eqref{eq-Psi-h-nonzero-sm} bounds, within a constant factor, the corresponding inner sum over $m\in\cJ$ in~\eqref{eq-sum-h-m}.  This is in contrast to the remaining terms (those in~\eqref{eq-Psi-h-nonzero-lg} and~\eqref{eq-Psi-m-notin-J}),
for which we must pay the extra factor $k^{2(n/k-h)}$ appearing in~\eqref{eq-Psibar-def}.

Before proving Lemma~\ref{lem-Psi-bounds}, we show that it gives
our earlier assertion that the terms with $h=0$ and $m\in\cJ$ dominate the sum in~\eqref{eq-sum-h-m};
that is,
\begin{equation}
  \label{eq-sum-Psi-bound}
  \sum_h\sum_{m}\match(n-2m)\sum_{S\in\inter_{h,m}}\sN(S)^2 \leq (1+o(1))\sum_{m\in\cJ}\Psi_0(m)\,.
\end{equation}
For $m\in\cJ$ and $1 \leq h \leq h^*$, setting $n'=n-kh$, we have
\[ \left|\frac{m}{n'} - \frac13\right| \leq \delta + O(kh / n) = \delta + O\left(1/\sqrt{\log n}\right) < (\log n')^{-1/3}\]
for large enough $n$. Theorem~\ref{thm-N(S)} thus says that $\sum_{S' \in \inter^h_{0,m}} \sN(S')^2 $ is at most $(9+o(1))\left(m! 2^m\right)^2$.
% and combining this with~\eqref
Hence, using~\eqref{eq-sum-inter-0m-h} (and slightly simplifying) we have
\begin{align*}
\sum_{h=1}^{h^*}\sum_{m\in\cJ}\match(n-2m)\sum_{S\in\inter_{h,m}}\sN(S)^2 %\\
%&= \sum_{h=1}^{h^*} \sum_{m\in\cJ} \match(n-2m)\binom{n/k}h^2 (2k)^h h! \sum_{S'\in\inter_{0,m}^h}\sN(S')^2
&\lesssim
\sum_{h=1}^{h^*} \sum_{m\in\cJ} \Psi_h(m)\,,
\end{align*}
which is $o\left(\sum_{m\in\cJ}\Psi_0(m)\right)$ by~\eqref{eq-Psi-h-nonzero-sm}.
For the terms in~\eqref{eq-sum-Psi-bound} corresponding to $h>h^*$ or $m\notin\cJ$, we can use the trivial bound $\sN(S')\leq k^{n'/k} m! 2^m$,
%(note that the definition of $\Psi_h$ accounts for rooting the $h$ cycles in the intersection pattern, hence a $k^{2hh}$ factor is overcounted)
whence~\eqref{eq-Psi-h-nonzero-lg} and~\eqref{eq-Psi-m-notin-J} imply a total contribution of $o\left(\sum_{m\in\cJ}\Psi_0(m)\right)$.
So, Lemma~\ref{lem-Psi-bounds} indeed allows us to estimate~\eqref{eq-sum-h-m}.
\begin{proof}[\textbf{\emph{Proof of Lemma~\ref{lem-Psi-bounds}}}]
For any $m$ and $h$, we have
%\begin{align}
%\frac{\Psi_{h}(m)}{\Psi_{h_0}(m)}  &= %\frac{(n-hk-m-1)_{m-1}}{(n-m-1)_{m-1}} \\ &=
%\frac{(2n^2/k)^{h-h_0}}{(h)_{h-h_0}}\prod_{i=1}^{m-1}\left(1-\frac{k(h-h_0)}{n-k h_0-m-i}\right)\nonumber\\
%&\leq \frac{(2n^2/k)^{h -h_0}}{(h)_{h-h_0}}\exp\left[-k(h-h_0)\log \frac{n-m-1}{n-2m}\right]\label{eq-Psi-h2-Psi-h1-ratio}\,.
%\end{align}
\begin{align}
\frac{\Psi_{h}(m)}{\Psi_{0}(m)}  &= %\frac{(n-hk-m-1)_{m-1}}{(n-m-1)_{m-1}} \\ &=
\binom{n/k}h^2 (2k)^h h! \prod_{i=1}^{m-1}\left(1-\frac{kh}{n-m-i}\right) \nonumber\\
&\leq \frac{(2n^2/k)^{h}}{h!}e^{-kh\log \frac{n-m-1}{n-2m}}
\leq \left(2e\frac{n^2}{kh}e^{-k\log \frac{n-m-1}{n-2m}}\right)^h
\label{eq-Psi-h2-Psi-h1-ratio}\,.
\end{align}
The two cases with $m\in\cJ$ (\eqref{eq-Psi-h-nonzero-sm} and~\eqref{eq-Psi-h-nonzero-lg}) are easy. Here we observe that, even with our lower bound
on $k$ relaxed to $k>(2+\epsilon)\log_2 n$ for some $\epsilon>0$, the right-hand side of~\eqref{eq-Psi-h2-Psi-h1-ratio} is at most
\[ \left[2e \frac{n^2}{kh} (2-o(1))^{-k} \right]^h < n^{-\left(\epsilon-o(1)\right)h}\,. \]
This gives~\eqref{eq-Psi-h-nonzero-sm} (actually the finer $\sum_{h=1}^{h^*}\Psi_h(m) = o(\Psi_0(m))$ for $m\in\cJ$).
Similarly,~\eqref{eq-Psi-h-nonzero-lg} follows once we observe that when $h>h^*$ and $m\in\cJ$ one has
$kh \log\frac{n-m-1}{n-2m} \gtrsim n/\sqrt{\log n}$. Thus, the term $\exp(-kh\log\frac{n-m-1}{n-2m})$ in~\eqref{eq-Psi-h2-Psi-h1-ratio} easily
overpowers the extra $k^{2(n/k-h)}$ in $\hat{\Psi}_h(m)$.

For the more interesting~\eqref{eq-Psi-m-notin-J}, we write $m=n/3+\gamma n$ and proceed as follows.
First suppose $m \geq \frac2{9} n$.
When estimating $\Psi_h(m)/\Psi_0(m)$ via~\eqref{eq-Psi-h2-Psi-h1-ratio},
the two opposing factors in that bound are %$\exp(- k h \log\frac{n-m-1}{n-2m})$, i.e.,
$\exp\big[-k h \log\frac{2-3\gamma-O(1/n)}{1-6\gamma}\big]$ vs.\
$\exp\big[h(2\log n - \log k + O(1))\big]$;
thus, if
\begin{equation}\label{eq-down-deviation-gamma}
k \log\frac{2-3\gamma-O(1/n)}{1-6\gamma} \geq (2+\epsilon) \log n
\end{equation}
for some fixed $\epsilon>0$ then $\Psi_h(m)/\Psi_0(m) \leq \exp[-O(n^{-\epsilon})]$.
Since $k \geq K_0(n)= 2\log_{4/3}(2n/e)$, it follows that
\[ \exp\left(2 \frac{\log n}k \right) \leq \exp\left(2 \frac{\log n}{K_0(n)} \right)= \frac43 -o(1)\,,\]
and so~\eqref{eq-down-deviation-gamma} easily holds as long as $\frac{2-\gamma}{1-6\gamma} \geq \frac43 +\epsilon'$
for some $\epsilon'>0$, and in particular it holds for any $\gamma \geq  -\frac19$ (with room to spare). Overall we have
\[ \sum_{h=1}^{n/k}\sum_{\substack{m \geq \frac29n \\ m\notin\cJ}} \hat{\Psi}_h(m) = o\bigg(\sum_{\substack{m \geq \frac29n \\ m\notin\cJ}}\hat{\Psi}_0(m)\bigg) =
o\bigg(\sum_{m\in\cJ}\Psi_0(m)\bigg)\,,\]
with the last equality using~\eqref{eq-Psi(m)-lower}--\eqref{eq-Psi(m)-upper},
in which the deviation of $\delta n$ in $m$, compared to the interval $\cJ$, translates into a bound of $\exp(-c n \log^{-2/3} n)$ for an absolute $c>0$ and overtakes the factor
$k^{2n/k}=\exp[O(n\frac{\log k}k)]$.

It remains to establish~\eqref{eq-Psi-m-notin-J} when $m < \frac29 n$.
For such values of $m$ we get from~\eqref{eq-Psi(m)-lower} that, for some absolute constant $c_0>0$,
\begin{equation}\label{eq-side-deviation}
\Psi_0(m)/\Psi_0(\lfloor\tfrac{n}3\rfloor) \leq
e^{-c_0 n}\,.
\end{equation}
Observe now that if $h < n/\log^2 n$ then $n^{2h} = \exp(O(n/\log n)) = \exp(o(n))$, whereas if $h\geq n/\log^2 n$,
\[ \frac{(2n^2/k)^h}{h!} = \left[(2e+o(1))\frac{n^2}{hk}\right]^h \leq e^{(1+o(1))h\log n} \,.\]
Immediately we see that if $k > (2/c_0) \log n$ (say) then the last expression is at most
$\exp[(c_0/2+o(1))n]$, outweighed by
the factor $\exp(-c_0 n)$ from~\eqref{eq-side-deviation};
Similarly, if $hk/n \leq 5c_0$ then $\frac{h\log n}n \leq \frac{hk \log n}{n K_0(n)} < \frac56c_0$
(using $K_0(n)\geq 6\log n$),
again outweighed by the aforementioned factor $\exp(-c_0n)$.
Setting
\[ \alpha := m/n\,,\quad \theta = h k / n\,,\]
it therefore remains to show~\eqref{eq-Psi-m-notin-J} when
\begin{equation}
  \label{eq-last-regime}
  \theta> c> 0\,,\quad \frac{2(1-\theta)}k\leq \alpha\leq \frac{1-\theta}2\;\wedge\;\frac29\,,\quad k \leq c' \log n
\end{equation}
for some absolute constants $c,c'>0$ (the upper bound on $\alpha$ used that every path has length at least 2,
while the lower bound on $\alpha$ used that each of the $(1-\theta)n/k$ cycles unaccounted for by $h$
 must contain at least 2 paths).

We next treat the range $c < \theta \leq 3/4$. With the above notation for $m$,
\begin{align}
&\Psi_0(m) = \match(n-2m)  m!2^m  \binom{n-m-1}{m-1} \nonumber\\
&\asymp \left(\frac{(1-2\alpha)n}e\right)^{(1-2\alpha)n/2}  \sqrt{\alpha n} \left(\frac{\alpha n}e\right)^{\alpha n} 2^{\alpha n}
 O\Big(\frac{1}{\sqrt{\alpha n}}\Big)e^{(1-\alpha)H_e\left(\frac{\alpha}{1-\alpha}\right)n}
\nonumber\\
&\asymp (n/e)^{n/2} \, e^{\left[ (1-\alpha)H_e\left(\frac{\alpha}{1-\alpha}\right)+\frac{1-2\alpha}2\log(1-2\alpha) + \alpha\log(2\alpha) \right] n}\,,
\nonumber
%\label{eq-psi0m-g(a)}
\end{align}
where $H_e(x)=-x\log x - (1-x)\log(1-x)$ is the natural entropy function, and using the fact $\binom{y}{\alpha y} \asymp (\alpha y)^{-1/2}
\exp(H_e(\alpha)y)$, valid for any $\alpha\in(0,1)$ and $y$. Letting
\begin{align}
  \label{eq-g-def}
  g(\alpha) := (1-\alpha)H_e\left(\frac{\alpha}{1-\alpha}\right)-\frac12 H_e(2\alpha)\,,
\end{align}
we see that
\begin{align}\label{eq-psi0m-g(a)}
\Psi_0(m) &\asymp (n/e)^{n/2} \exp\left[ g(\alpha) n\right]\,.
\end{align}
Combining this with~\eqref{eq-Psi-h2-Psi-h1-ratio} (using $h k \log(\frac{n-m-1}{n-2m}) = hk\log(\frac{n-m}{n-2m})+O(1)$) gives
\begin{align}
  \frac{\hat{\Psi}_h(m)}{\Psi_0(\lfloor\tfrac{n}3\rfloor)} &=
  \frac{\hat{\Psi}_h(m)}{\Psi_0(m)}\frac{\Psi_0(m)}{\Psi_0(\lfloor\tfrac{n}3\rfloor)}\nonumber\\
  &\lesssim k^{2n/k}\exp\left[(\theta n/k) (\log n + O(1)) - \theta n \log\frac{1-\alpha}{1-2\alpha} \right]\nonumber\\
  & \qquad\quad\cdot\exp\Big[\left(g(\alpha)-g(1/3)\right)n\Big]\,.\label{eq-Psih-Psi0(n/3)-ratio}
\end{align}
If $\log\frac{1-\alpha}{1-2\alpha} \geq \frac{\log n}k$ then the first exponent in the final expression is at most $\exp(o(n))$ and the entire
expression is less than $\exp(-c n)$ for some fixed $c>0$ thanks to the term $\exp[(g(\alpha)-g(1/3))n]$ (as $\alpha \leq \frac29$ and $g'(\alpha) = \log\frac{2-4\alpha}{1-\alpha}$, so $g$ is increasing for $\alpha\leq\frac13$).
Assume therefore that
\[ \log\frac{1-\alpha}{1-2\alpha} < \frac{\log n}k\,.\]
Then the right-hand side of~\eqref{eq-Psih-Psi0(n/3)-ratio} is increasing in $\theta$
and decreasing in $k$, so we can take $\theta = \frac34$ (its maximum in our present regime) and $k=K_0(n)$ to get
\begin{align}\label{eq-Psih-Psi0(n/3)-ratio-2}
  \frac{\hat{\Psi}_h(m)}{\Psi_0(\lfloor\tfrac{n}3\rfloor)}
\lesssim k^{2n/k}e^{\left[\frac34 \frac{\log n + O(1)}{K_0(n)} - \frac34 \log\frac{1-\alpha}{1-2\alpha} + g(\alpha)-g(\frac13) \right]n}\,.
\end{align}
Since $g(1/3) = \frac12\log(4/3)$ and $K_0(n)=2\log_{4/3}(2n/e)$, clearly
\begin{equation}
  \label{eq-logn-K0-g(1/3)}
  \frac{\log n}{K_0(n)} = g(1/3) + O(1/\log n)\,,
\end{equation}
and we further claim that
\begin{equation}
  \label{eq-g(a)-vs-log}
  \beta\log\frac{1-\alpha}{1-2\alpha} > g(\alpha)\qquad\mbox{ for any $\alpha\in(0,\frac12)$ and $\beta > \log 2$}\,.
\end{equation}
Indeed, letting $f(\alpha)$ denote the left-hand side, we have $f(0)=g(0)=0$, so~\eqref{eq-g(a)-vs-log} will follow from showing that $f'(\alpha)> g'(\alpha)$ for $\alpha\in(0,1/2)$. Along this interval $f'(\alpha)=\beta/(1-3a+2a^2)$ is increasing,
while $g'(\alpha)=\log(\frac{2-4\alpha}{1-\alpha})$ is decreasing,
and using $f'(0)=\beta > \log 2 = g'(0)$ this implies~\eqref{eq-g(a)-vs-log}.

Plugging~\eqref{eq-logn-K0-g(1/3)} and~\eqref{eq-g(a)-vs-log} (with $\beta=3/4$) in~\eqref{eq-Psih-Psi0(n/3)-ratio-2} yields
\[   \frac{\hat{\Psi}_h(m)}{\Psi_0(\lfloor\tfrac{n}3\rfloor)}
\lesssim k^{2n/k}e^{- (\frac14-o(1)) g(\frac13) n}\,,
\]
which clearly suffices to show that
\[ \sum_{h\leq \frac34 n/k}\sum_{m\notin\cJ} \hat{\Psi}_h(m) = o\bigg(\sum_{m\in\cJ}\Psi_0(m)\bigg)\,.\]

We proceed to the case $\frac34 \leq \theta \leq 1-k/n$.
Here we will compare $\Psi_h(m)$ to $\Psi_{n/k}(0)$ (rather than $\Psi_{0}(\lfloor\tfrac{n}3\rfloor)$),
and later show that $\Psi_{n/k}(0)$ is small.
Observe that
\begin{align*}
\frac{\hat{\Psi}_h(m)}{\hat{\Psi}_{n/k}(0)} &= \frac{\match(n-2m)}{\match(n)} m! 2^m \binom{n-kh-m-1}{m-1}
\frac{\binom{n/k}{h}^2 h! (2k)^{2n/k-h}}{(n/k)! (2k)^{n/k}}  \\
&\lesssim \left(\frac{n-2m}{e}\right)^{\frac{n-2m}2}\left(\frac{n}e\right)^{-\frac{n}2} 2^m (n-kh-m)^m \frac{\binom{n/k}{h} }{(n/k-h)! } (2k)^{n/k-h}\,,
\end{align*}
using $m \leq n-kh-m$. Writing $\binom{n/k}h \leq (\frac{en}{k(n/k-h)})^{n/k-h}$, we find the last expression to be at most
\begin{align*}
& O(1)\left(1-\frac{2m}n\right)^{n/2}\left(\frac{2e(n-kh-m)}{n-2m}\right)^m \frac{\big[\frac{e}{1-\theta}\big]^{(1-\theta) \frac{n}k}
\big[\frac{2ek^2}{(1-\theta)n}\big]^{(1-\theta)\frac{n}k} }
{\sqrt{(1-\theta)\frac{n}k} }  \\
&\qquad\qquad \leq  \left(\frac{2(1-\theta-\alpha)}{1-2\alpha}\right)^{\alpha n} \left[\frac{2e^2k^2}{(1-\theta)^2n}\right]^{(1-\theta)\frac{n}k}
 \,,
\end{align*}
using $\sqrt{(1-\theta)n/k} \geq 1$ since $\theta \leq 1- k/n$.
For any $\theta >1/2$ we know that $\frac{2(1-\theta-\alpha)}{1-2\alpha}$ is at most $2(1-\theta)$. Recalling from~\eqref{eq-last-regime} that
$\alpha n \geq 2(1-\theta)n/k$, we can infer that the right-hand side of the last inequality
is at most
\[  \left(\frac{8e^2k^2}{n}\right)^{(1-\theta)\frac{n}k}
\,.
\]
Since $\theta = ik/n$ for some integer $1\leq i \leq \frac14n/k$, in which case $m \leq ik/2$,
summing over the last expression over $m$ and $\theta$ amounts to
\[ \frac12 \sum_{i=1}^{\frac14n/k} ik\left(\frac{8e^2 k^2}n\right)^i = O(k^3/n)\]
(here we used the fact $k = O(\log n)$ from~\eqref{eq-last-regime}).

To complete the proof, we analyze $h=n/k$. By~\eqref{eq-g-def}--\eqref{eq-psi0m-g(a)}, together with the fact that $g(1/3)=\frac12\log(4/3)$,
\begin{align*}
 \frac{\hat{\Psi}_{n/k}(0)}
{\Psi_0(\lfloor\tfrac{n}3\rfloor)} &\asymp \frac{(n/e)^{n/2} (n/k)! (2k)^{n/k}}{(n/e)^{n/2}\exp[g(\frac13)n] }
\asymp \sqrt{n/k}\left(\frac{2n}{e}\right)^{\frac{n}k} \left(\frac{3}4\right)^{\frac{n}2} \leq \sqrt{n/k}\,,
\end{align*}
using the fact $(2n/e)^{n/k} \leq (4/3)^{n/2}$ since $k \geq K_0(n)=2\log_{4/3}(2n/e)$.
Recalling from~\eqref{eq-middle-sum} that $\sum_{m\in\cJ} \Psi_0(m) \asymp \sqrt{n}\Psi_0(\lfloor\tfrac{n}3\rfloor)$ now gives
\[ \sum_{h\geq \frac34 n/k}\sum_{m\notin\cJ} \hat{\Psi}_h(m) = O(1/\sqrt{k}) \sum_{m\in\cJ}\Psi_0(m)\,.\]
This establishes~\eqref{eq-Psi-m-notin-J} and completes the proof of the lemma.
\end{proof}

To complete the proof of the theorem, we revisit the definition of $\Psi_0$ in~\eqref{eq-Psi-def}, and see that $\Psi_0(\lfloor\tfrac{n}3\rfloor)$ can readily be estimated as
\begin{align*} \Psi_0( \tfrac{n}3) &\sim \sqrt{2\pi n}(n/e)^n (n/3) 2^{n/6}
\frac{\sqrt{\frac43\pi n}\left(\frac23n -1\right)^{\frac23 n -1}}
{\sqrt{\frac13\pi n}\left(n/6\right)^{n/6}}e^{1-n/2} \\
&\sim \sqrt{2\pi n} (n/e)^{3n/2} \left[ (12)^{\frac16} (\tfrac23)^{\frac23} \right]^{n}
= \sqrt{2\pi n} (n/e)^{3n/2} (4/3)^{n/2}\,,
\end{align*}
where we used the fact that $(\frac23 n-1)^{\frac23 n-1} \sim (\frac23 e n)^{-1}(\frac23 n)^{\frac23 n}$.
Combining Lemma~\ref{lem-Psi-bounds} with~\eqref{eq-middle-sum} yields
\[ \sum_h \sum_m \match(n-2m) \sum_{S \in \inter_{h,m}} \sN(S)^2 \leq \left(6\pi \sqrt{2} + o(1)\right) n \Big(\frac{n}e\Big)^{3n/2}  \Big(\frac43\Big)^{n/2}\,,\]
and altogether we have
\begin{align*}
\E[Y_k^2] &= \frac{6^n}{\match(3n)} \sum_{m} \match(n-2m) \sum_{S \in \inter_{h,m}} \sN(S)^2\\
&\leq 6\pi\sqrt{2} \frac{6^n}{\sqrt{2}(\frac{3n}e)^{\frac{3n}2}} n \Big(\frac{n}e\Big)^{3n/2} \Big(\frac43\Big)^{n/2} = 6 \pi n \frac{6^n}{3^{3n/2}} \Big(\frac43\Big)^{n/2} = 6\pi n \Big(\frac43\Big)^{n} \,.
\end{align*}
By~\eqref{eq-E[Y]} we have thus arrived at the promised
\[ \E[Y_k^2] = (3+o(1))(\E Y_k)^2\,,\]
as required.
\end{proof}

\section{Small subgraph conditioning and contiguity}\label{sec-small-subgraph}
In this section we derive Theorem~\ref{mainthm-cycle-factor} and Corollary~\ref{maincoro-contig}
from Theorem~\ref{thm-second-moment}, which was the upshot of the second moment analysis of the previous sections. For this we use what is called the small subgraph conditioning method in~\cite{Wormald}. This calls for estimating the joint moments $\E\big[Y_k\prod_{i=1}^j[X_i]_{r_i}\big]$ where $X_i$ is the number of $i$-cycles in a pairing in $\cP_{n,3}$ (and $[X]_{r}=X!/(X-r)!$) for each fixed vector $(r_1,\ldots, r_j)$. The computation here is almost the same as the corresponding original computation, the first time small subgraph conditioning was used, when the random variable of concern was the number of Hamilton cycles, and this is presented for the configuration model in~\cite{Wormald}*{Proof of Theorem 4.5}. The argument is short so we include a complete version here.

We first show that for any  fixed  $i\ge 1$
\begin{equation} \label{onemoment}
\frac{\E[Y_kX_i]}{\E[Y_k]}\to \lambda_i(1+\delta_i)
\end{equation}
where
\[
\lambda_i =\frac{2^i}{2i}\quad \delta_i = \frac{(-1)^i -1}{2^i}.
\]
The fact that $\lambda_i = \E[X_i]$ was one of the first results
Let $D$ be some fixed set of pairs forming
a $\rod$ cycle factor in pairings in $\cP_{n,3}$. By symmetry all copies
of $D$ are
equivalent and  so
\[
 \E [Y_kX_i]/\E[ Y_k] = \E\left[X_i\mid D\subseteq \cP_{n,3}\right]\,.
 \]

If $C$ is the set of pairs corresponding to an $i$-cycle (in which
case we also call $C$ itself an $i$-cycle), since $k\to\infty$ and $i$ is fixed, we can assume $ D\cap C$ forms a  configuration of paths. We will classify
$C$ according to this configuration.  Give these
paths a consistent orientation along $C$  and distinguish one path as first. This induces a linear ordering of paths around $C$, and considerting the overcounting we just introduced, it is now clear that
\begin{equation} \label{Qsum}
\frac{\E[Y_kX_i]}{\E[Y_k]} =
\sum_{Q}\frac{1}{2|Q|}\E[X_i(Q)\mid D\subseteq \cP_{n,3}]
\end{equation}
where $Q$ denotes the sequence of lengths
of paths, $|Q|$ is the number of paths in $Q$ and $X_i(Q)$ is the number of
$i$-cycles in $\cP_{n,3}$ consistent with such
a configuration $Q$. Fix on such a $Q$ with $|Q|=j$. There are
asymptotically
$n^j$ ways to choose the starting
points of the paths on $D$ together with their directions along $D$.
Almost all choices of such starting points are such that each two starting points are at distance greater than $i$ along the cycles in $D$. Furthermore, once they
are chosen, the pairs in $C$ are determined if it creates an $i$-cycle yielding
$Q$. The probability that these pairs all occur in
$\cP_{n,3}$  conditional upon $D\subseteq P$ is asymptotically $n^{-j}$. Hence
$
\E[X_i(Q)\mid D\subseteq P] \to 2^{|Q|}.
$
Now~\eqref{Qsum} becomes
\begin{equation} \label{Qsum2}
\frac{\E[Y_kX_i]}{\E[Y_k]} \to
\sum_{j\ge 1}\frac{2^j}{2j}\ |\,\{Q\,\colon\, |Q|=j\}\,|.
\end{equation}
Note that every $Q$ must have $i$ vertices in total.

The ordinary generating function for the number of configurations $Q$
with $x$ marking the total number of vertices involved and $y$ marking
the number of paths is $\frac{g(x,y)}{1-g(x,y)}$ where $g(x,y)$ is the
generating function for one path; that is, $ yx^2/(1-x)$. Thus,
with square brackets denoting extraction of coefficients,
\[
\frac{\E[Y_kX_i]}{\E[Y_k]} \to
\sum_{j\ge 1}\frac{2^j}{2j}[x^iy^j] \frac{yx^2}{1-x-yx^2}
\]
and standard generating function manipulations (shown in detail in~\cite{Wormald})
now give~\eqref{onemoment}. This argument is easily generalised to give
\[
\E\big[Y_k\prod_{i=1}^j[X_i]_{r_i}\big]\to \prod_{i=1}^j\lambda_i^{r_i}(1+\delta_i)^{r_i} \,.
\]

By definition of $\lambda_i$ and $\delta_i$ we have  $\sum_{i\ge
1}\lambda_i\delta_i^2 = \log 3$. Since we showed in Theorem~\ref{thm-second-moment} that
$\E Y_k^2 \leq (3+o(1))(\E Y_k)^2$,
all the requirements of small subgraph conditioning are satisfied
(see~\cite{Wormald}*{Theorems~4.1 and 4.3} or~\cite{JLR}*{Theorem 9.12 and Remark 9.16}) provided that $\E[Y_k]\to\infty$.
To state the conclusion of this, we first note that $\delta_i=-1$ iff $i=1$. Several things may now be concluded when $\E Y_k \to\infty$. One is that $\cP_{n,3}$, conditioned on no loops, is contiguous to the superposition of a random $\rod$ factor and a random perfect matching. Another is that in the same conditional space, the variable $\cfact$ converges to a distribution of the type stated for $W$ given in Theorem~\ref{mainthm-cycle-factor}, but with the product starting at $j=2$ rather than $j=3$. The convergence in that statement and of the $X_i$ to Poisson with means $\lambda_i$ hold jointly. Hence, conditioning on $X_2=0$, i.e., the multigraph has no multiple edges either, we obtain the statement in the theorem (as $\P(X_2=0)$ is bounded away from 0).

\section{The Comb Conjecture via cycle factors in regular graphs}\label{sec-gnp-comb}
In this section we derive the proof of Theorem~\ref{mainthm-comb} using our main result on cycle factors in random cubic graphs, Theorem~\ref{mainthm-cycle-factor}. As a preliminary step we establish the next corollary
on the threshold for a $k$-cycle factor in $\cG(n,p)$.
\begin{corollary}
  \label{cor-cycle-factor-gnp}
Fix $\epsilon > 0$. For any $k,n$ such that $k\mid n$ and $k \geq K_0(n)$ as given in~\eqref{eq-K0},
the random graph $\cG(n,p)$ with $p=(2+\epsilon)\frac{\log n}n$ has a $k$-cycle factor
w.h.p. In particular, for any $k \geq K_0(n)$ dividing $n$, the threshold for the existence of a $k$-cycle factor in $\cG(n,p)$ is at $p\asymp \frac{\log n}n$.
\end{corollary}
\begin{proof}
To simplify the exposition, we first give a proof establishing this fact for $p=(3+\epsilon)\frac{\log n}n$. Further note that since the threshold for connectivity in $\cG(n,p)$ is at the edge-probability $p=(1+o(1))\frac{\log n}n$, the above fact is already sufficient for establishing $p\asymp \frac{\log n}n$ as the threshold for a $k$-cycle factor in $\cG(n,p)$.

Fix $\epsilon>0$ and consider the random graphs $G_i \sim \cG(n,p')$ for $i=1,2,3$, where $p'=(1+\epsilon/3)\frac{\log n}n$. Clearly, by stochastic domination, for any given {\em simple} graph $F$ the probability that $G'\sim \cG(n,3p')$ contains $F$ as a subgraph is at least the probability that
$F$ appears as a subgraph of the multigraph $H$ comprised of the union of $G_1,G_2,G_3$.

 Since the threshold for the appearance of a perfect matching in $\cG(n,p)$ is at $p=(1+o(1))\frac{\log n}n$, w.h.p.\ we can extract an independent uniform perfect matching on $M_i$ from each of the $G_i$'s and denote the multigraph formed by the union of $M_1,M_2,M_3$ by $H_0 \subset H$.

By the contiguity results of~\cites{Janson,MRRW}, it is known that the multigraph $H_0$ is contiguous to a uniformly chosen 3-regular multigraph $\cP_{n,3}$ conditioned to have no loops (see~\cite{JLR}*{Theorem 9.40}, as well as~\cite{Wormald} for further information). We can therefore invoke Theorem~\ref{mainthm-cycle-factor} (and the remark following that theorem concerning multigraphs) and gather that $H_0$ contains a $k$-cycle factor with high probability. Carrying this to $G' \sim \cG(n,3p')$ concludes the proof.

To obtain the same result for $p=(2+\epsilon)\frac{\log n}n$, rather than using three uniform independent perfect matchings we instead take a uniform perfect matching $\mathcal{M}$ and an independent uniform Hamilton cycle $\mathcal{H}$, each of which has a threshold of $(1+o(1))\frac{\log n}n$ in $\cG(n,p)$. A delicate point one should note is the following: it is known that a random $3$-regular graph is contiguous to $\mathcal{H}\oplus\mathcal{M}$, the union of $\mathcal{H}$ and $\mathcal{M}$ conditioned on having no self-loops or multiple edges. However, in our case we wish to address the multigraph formed by the union of $\mathcal{H}$ and $\mathcal{M}$ conditioned only to have no self-loops. The fact that this multigraph is contiguous to a random 3-regular multigraph $G \sim \cP_{n,3}$ conditioned on having no self-loops (addressed by our cycle factor results) similarly holds (for example, see the second-to-last conclusion in the proof~\cite{Wormald}*{Theorem~4.5}) and follows from the exact same arguments in the framework of~\cite{FJMRW}, as well as from~\cite{Janson}*{Theorem 3}.
\end{proof}

\begin{proof}[\textbf{\emph{Proof of Theorem~\ref{mainthm-comb}}}]
As already mentioned, a prerequisite for containing the comb as a spanning subgraph is connectivity, whose threshold in $\cG(n,p)$ is at $p = (1+o(1))\frac{\log n}n$; hence, establishing the threshold for the appearance of the comb will readily follow from showing that w.h.p.\ $\cG(n,p)$ contains the comb at $p=(2+\epsilon)\frac{\log n}n$.

First consider $G' \sim \cG(n,p')$ for $p' = c/n$ with some large $c>0$. By the results of Ajtai, Koml{\'o}s and Szemer{\'e}di~\cite{AKS}, as well as de la Vega~\cite{dlVega1}, in this regime w.h.p.\ the random graph contains a path of length $c' n$ for some absolute $c'>0$. In particular, there exists some path $P$ of length $\sqrt{n}$ w.h.p.; let $M$ denote the remaining $m=n-\sqrt{n}$ vertices. The path $P$ will serve as the spine of the comb whereas the vertices of $M$ will produce its $\sqrt{n}$ teeth.

Fix $\epsilon>0$ and consider the random graph $G''\sim \cG(n,(2+\epsilon/2)\frac{\log n}n)$. We first examine the induced subgraph on the vertices of $M$, which is a random graph $\cG(m,p'')$ with $p'' = (2+\epsilon/2-o(1))\frac{\log m}m$. Applying Corollary~\ref{cor-cycle-factor-gnp} for $k = \sqrt{n}-1 = (1-o(1))\sqrt{m}$ we deduce that w.h.p.\ this random graph contains a $k$-cycle factor. In other words, w.h.p.\ we can partition the vertices of $M$ into $m/k=\sqrt{n}$ disjoint $k$-cycles, which we denote by $C_1,\ldots,C_{\sqrt{n}}$.

Observe that in our analysis of $G''$ we have thus far only addressed the edges within the induced subgraph on $M$. Of the remaining edges, every edge between $P$ and $M$ appears in $G''$ independently with probability $p''$. Therefore, the random bipartite graph whose sides are the vertices of $P$ vs.\ the cycles $C_1,\ldots,C_{\sqrt{n}}$, with an edge connecting a vertex $u\in P$ and the cycle $C_i$ iff they are connected in $G''$ as above, has edge probability $1-(1-p'')^k = (2-o(1))\frac{\log \sqrt{n}}{\sqrt{n}}$. By results of Erd\H{o}s and R\'enyi~\cite{ER66} this edge probability is asymptotically twice than the threshold for a perfect bipartite matching in this bipartite random graph; thus, w.h.p.\ we can match each vertex of $P$ to an exclusive $(\sqrt{n}-1)$-cycle.

Unraveling the cycles in the obvious manner now produces the comb (with $P$ as its backbone). Altogether, the comb appears as a spanning subgraph of $\cG(n,p)$ with $p=p'+p''=(2+\epsilon/2+o(1))\frac{\log n}n$, as required.
\end{proof}

\begin{remark}\label{rem-gen-comb}
The analogue of Theorem~\ref{mainthm-comb} holds for generalized combs $\combg_{n,k}$ (in which the spine has $n/k$ vertices, and the teeth are of length $k$) as long as $k \geq K_0(n)$.
Indeed, via the exact same proof, we ultimately seek a perfect matching between $n/k$ vertices (the spine) and $k$-cycles, which exists w.h.p.\ since the edge-probability in this bipartite graph is $(1+o(1))p k \geq (2+o(1))\frac{\log (n/k)}{n/k}$.
\end{remark}
%\section{Concluding remarks and open problems}

\subsection*{Acknowledgment}
A major part of this work was carried out while the first and third authors were visiting the Theory Group of Microsoft Research, Redmond. They would like to thank the Theory Group for its hospitality and for creating a stimulating research environment.

\end{document}